\documentclass{amsart}

\usepackage{amssymb}
\usepackage[dvipsnames]{xcolor}
\usepackage[all,cmtip]{xy}

\usepackage{graphicx}
\usepackage{setspace}
\usepackage{tikz}
\usepackage{tkz-graph}
\usetikzlibrary{arrows}

  \makeatletter
  \CheckCommand*\refstepcounter[1]{\stepcounter{#1}%
      \protected@edef\@currentlabel
       {\csname p@#1\endcsname\csname the#1\endcsname}%
  }
  \renewcommand*\refstepcounter[1]{\stepcounter{#1}%
    \protected@edef\@currentlabel
      {\csname p@#1\expandafter\endcsname\csname the#1\endcsname}%
  }
  \def\labelformat#1{\expandafter\def\csname p@#1\endcsname##1}
  \DeclareRobustCommand\Ref[1]{\protected@edef\@tempa{\ref{#1}}%
     \expandafter\MakeUppercase\@tempa
  }
  \makeatother

  \makeatletter
  \newcommand{\numberlike}[2]{%
     \expandafter\def\csname c@#1\endcsname{%
         \expandafter\csname c@#2\endcsname}%
  }
  \makeatother

  \def\DefaultNumberTheoremWithin{section}

  \theoremstyle{plain}
  \newtheorem{Lemma}{Lemma}
     \numberwithin{Lemma}{\DefaultNumberTheoremWithin}
     \labelformat{Lemma}{Lemma~#1}
  
     \numberwithin{Claim}{\DefaultNumberTheoremWithin}
     \numberlike{Claim}{Lemma}
     \labelformat{Claim}{Claim~#1}
  \newtheorem{Theorem}{Theorem}
     \numberwithin{Theorem}{\DefaultNumberTheoremWithin}
     \numberlike{Theorem}{Lemma}
     \labelformat{Theorem}{Theorem~#1}
  \newtheorem{Corollary}{Corollary}
     \numberwithin{Corollary}{\DefaultNumberTheoremWithin}
     \numberlike{Corollary}{Lemma}
     \labelformat{Corollary}{Corollary~#1}
  \newtheorem{Proposition}{Proposition}
     \numberwithin{Proposition}{\DefaultNumberTheoremWithin}
     \numberlike{Proposition}{Lemma}
     \labelformat{Proposition}{Proposition~#1}
  \newtheorem{Conjecture}{Conjecture}
     \numberwithin{Conjecture}{\DefaultNumberTheoremWithin}
     \numberlike{Conjecture}{Lemma}
     \labelformat{Conjecture}{Conjecture~#1}
     
  \theoremstyle{definition}
  \newtheorem{Definition}{Definition}
     \numberwithin{Definition}{\DefaultNumberTheoremWithin}
     \numberlike{Definition}{Lemma}
     \labelformat{Definition}{Definition~#1}
  \theoremstyle{definition}
  
     \numberwithin{Question}{\DefaultNumberTheoremWithin}
     \numberlike{Question}{Lemma}
     \labelformat{Question}{Question~#1}
     
  \theoremstyle{definition}
  
     \numberwithin{Problem}{\DefaultNumberTheoremWithin}
     \numberlike{Problem}{Lemma}
     \labelformat{Problem}{Problem~#1}

  \theoremstyle{remark}
  
     \numberwithin{Remark}{\DefaultNumberTheoremWithin}
     \numberlike{Remark}{Lemma}
     \labelformat{Remark}{Remark~#1}
  \theoremstyle{remark}

     \numberwithin{Example}{\DefaultNumberTheoremWithin}
     \numberlike{Example}{Lemma}
     \labelformat{Example}{Example~#1}
  
     \labelformat{Case}{Case~#1}
     \numberwithin{Case}{Lemma}
  
     \labelformat{Step}{Step~#1}
     \numberwithin{Step}{Lemma}

  \labelformat{equation}{(#1)}
  \labelformat{figure}{Figure~#1}
  \labelformat{section}{\S#1}
  \labelformat{subsection}{\S#1}
  \labelformat{subsubsection}{\S#1}

  \def\eqref{\ref}

  \def\ZZ{\mathbb{Z}}
   \def\QQ{\mathbb{Q}}
    \def\RR{\mathbb{R}}
  
  \def\Ker{\mathrm{Ker}}
  \def\Image{\mathrm{Im}}

  \def\Hom{\mathcal{H}}
  
  \def\Cube{{\normalfont {\textrm{Cube}}}}
  \def\Sing{{\normalfont {\textrm{Sing}}}}
  \def\Cell{{\normalfont {\textrm{Cell}}}}
  \def\CC{{\mathcal{C}}}

  \def\LL{{\mathcal{L}}}
   \def\DD{{\mathcal{D}}}
   \def\kK{{\mathcal{K}}} 

  \def\tilsig{\tilde{\sigma}}

\begin{document}
 
 \title 
{On the Vanishing of Discrete Singular Cubical Homology for Graphs}

\author[Barcelo]{H\'{e}l\`{e}ne Barcelo}

\address[H\'{e}l\`{e}ne Barcelo]
{
The Mathematical Sciences Research Institute, 17 Gauss Way, Berkeley, CA 94720, USA
}
\email{hbarcelo@msri.org}

\author[Greene]{Curtis Greene}

\address[Curtis Greene]
{
Haverford College, Haverford, PA 19041, USA
}
\email{cgreene@haverford.edu}

\author[Jarrah]{Abdul Salam Jarrah}

\address[Abdul Jarrah]
{
Department of Mathematics and Statistics, American University of Sharjah, PO Box 26666, Sharjah, United Arab Emirates
}
\email{ajarrah@aus.edu}

\author[Welker]{Volkmar Welker}

\address[Volkmar Welker]
{
Fachbereich Mathematik und Informatik, Philipps-Universit\"at, 35032 Marburg, Germany
}
\email{welker@mathematik.uni-marburg.de}

\thanks{This material is based upon work supported by the National Science Foundation under Grant No. DMS-1440140 while the authors were in residence at the
Mathematical Sciences Research Institute in Berkeley, California, USA}
\begin{abstract}
 We prove that if $G$ is a graph without 3-cycles and 4-cycles, then the discrete cubical homology of $G$ is trivial in dimension $d$, for all $d\ge 2$. We also construct a sequence $\{G_d\}$ of graphs such that this homology is
 non-trivial in dimension $d$ for $d\ge 1$. Finally, we show that  the discrete cubical homology
 induced by certain coverings of $G$ equals the ordinary singular homology of a $2$-dimensional cell complex built from $G$, although in general it differs from the discrete cubical homology of the graph as a whole.
\end{abstract}
\date{\today}

\maketitle

\section{Introduction}

We will be concerned with properties of a discrete (singular) cubical homology theory $\Hom^\Cube(G)$ for undirected graphs $G$, originally defined by Barcelo, Capraro, and White in \cite{BCW} for general 
metric spaces. In this paper we develop a discrete subdivision tool, facilitating computation and leading to the proofs of several conjectures made in \cite{BGJW}.

The homology theory defined  in \cite{BCW} had its roots in earlier work of Barcelo, Kramer, Laubenbacher, and Weaver \cite{BKLW}, which introduced a bi-graded family of discrete homotopy groups, $A_n^q(\Delta, x_0)$  for simplicial complexes. 
Graphical  (or equivalently, $1$-dimensional) versions of this homotopy theory were later studied by several authors including Babson, Barcelo, de Longueville, and Laubenbacher  \cite{BBLL}, and also Grigor'yan, Lin, Muranov, and Yau \cite{GLMY2}. In 2006, the authors of  \cite{BBLL} proposed the problem of finding a corresponding discrete homology theory, and  \cite{BCW} provided a solution in 2014.  Another homology theory for digraphs was defined by Grigor'yan et.\ al.\ in \cite{GLMY2}. However, \cite{BGJW} showed that, although this theory agrees with that of \cite{BCW} in dimension $1$, the two theories can differ in higher dimensions.

Aside from theoretical interest, motivation for studying discrete homotopy and discrete homology of graphs also comes from applications in pure and applied mathematics. In 1973, while studying base exchange graphs of matroids, Maurer \cite{Mau} proposed a new notion of fundamental groups for graphs, $\pi_1^*(G).$ As it turns out, for any graph $G$, 
Maurer's $\pi_1^*(G)$ is isomorphic to $A_1^1(G).$
A physicist, Atkin \cite{Atkin1974, Atkin1976}, developed similar ideas (also in dimension $1$) for applications to network analysis. Both Atkin and Maurer used their theory to measure a form of connectivity of graphs pertinent to their applications but different from standard connectivity measures in graph theory or algebraic topology. Recently,  the theory has found application, among others, in the study of complements of subspace arrangements (Barcelo, Severs and White \cite{BSW}), in coarse geometry (Delabie and Khukhro \cite{Khukhro,Vigolo},  
and in finite metric spaces (Rieser \cite{Rieser}).  More material on real world applications can be found in \cite{KraetzlLaubenbacher}.

In much of the work just described, the  3- and 4-cycles in a graph $G$ play a special role.  For example, in \cite{BKLW} it is shown that for any graph, the discrete fundamental group $A_1^1(G)$ is isomorphic to the ordinary fundamental group $\pi_1(K(G))$ of the cell complex obtained by attaching 2-cells to the 3- and 4-cycles of $G$, viewed as a 1-complex. This construction also appears in work of Lov\'asz \cite{Lovasz}, who used topological arguments to obtain results about 
connectivity of graphs. For a discussion  placing this work in the more general context of topological methods in
combinatorics, see \cite[Section 6]{Bj}. Our first main result (\ref{maintheorem}) will also feature $3$- and $4$-cycles in an essential way.

In \cite{BGJW}, the authors developed tools for computing the discrete singular cubical homology
groups $\Hom_d^\Cube(G)$ 
(see Section 2 of this paper for a precise definition)
for many families of graphs $G$. Since relevant chain groups grow super-exponentially in rank, 
direct computation presents formidable difficulties, even for small graphs.  For example, when $G  = \ZZ_5$, the pentagon (5-cycle) graph, the methods 
of \cite{BGJW} were not sufficient to determine 
$\Hom_d^\Cube(G)$ beyond $d=3$.  
In the present paper we prove that $\ZZ_5$ has vanishing discrete singular cubical homology in dimension $d$, for
all $d\ge 2$. This result was conjectured in \cite{BGJW}. In fact we will prove a stronger conjecture, also made in \cite{BGJW}.

\begin{Theorem}\label{maintheorem}
Let $G$ be any graph containing no $3$-cycles or $4$-cycles.  Then $\Hom_d^\Cube(G)=(0)$ for all $d\ge 2$.
\end{Theorem}

Of course, in classical theory, where graphs are $1$-dimensional CW-complexes, all graphs have trivial homology in dimension $d\ge 2$.  
However, the discrete cubical homology of \cite{BCW} is notably different: for example, $4$-cycles are homologically 
trivial in all dimensions.  
Examples of graphs with non-vanishing homology in dimension $d\ge 2$ exist (e.g., \cite{BCW}, \cite{BGJW}), but can be challenging to construct and verify. 
In that spirit, we note that the following conjecture appears in \cite{BGJW}, and remains open.

\begin{Conjecture}\label{conjecture1}
For any graph $G$, there exists an integer $N$ such that $\Hom_d^\Cube(G)= (0)$ for all $d \ge N$.
\end{Conjecture}

For arbitrarily large $d$, it is not obvious that there should exist graphs with non-vanishing $d$-homology.  
Proposition 5.3 of \cite{BCW} constructs an infinite sequence of graphs 
$\{G_d\}_{d \ge 1} $, such that 
\begin{eqnarray}
\label{recur}
\Hom_{d+1}^\Cube(G_{d+1})= \Hom_{d}^\Cube(G_{d}) \textrm{ for all $d\ge 1$}.
\end{eqnarray}
For appropriately chosen $G_1$ this yields a sequence of graphs with non-vanishing $d$-homology, for arbitrarily large $d$. 
However, the argument in \cite{BCW} is not sufficient to prove \Ref{recur} in full generality, since it relies on a discrete version of the Mayer-Vietoris sequence whose 
hypotheses are not satisfied for $d>1$.  In the present paper we will show that a small modification of the definition of $G_d$ 
in \cite{BCW} makes the argument correct, thus yielding a (different) sequence of graphs with non-vanishing homology in arbitrarily high dimension.

Our principal tool in proving \Ref{maintheorem} is a subdivision map that allows computation of homology to be restricted to ``small'' singular cubes. 
This approach is standard
in classical treatments of singular homology, see e.g., \cite{Massey91}, but for the discrete cubical case 
the details are significantly different and new techniques are required. 

The paper is organized as follows. Section 2 reviews the basic definitions of discrete cubical homology for graphs, following \cite{BCW}.  
Section 3 defines the subdivision map and proves the above mentioned conjecture when $G = \ZZ_5$.  
In fact the argument proves the same result for any cycle $\ZZ_n$ with $n\ge 5$.  
Section 4 introduces the machinery necessary to extend the proof for $G = \ZZ_5$ to arbitrary graphs without $3$-cycles and $4$-cycles. 
The key step is to show that every singular cube with codomain $G$ can be lifted to the universal covering graph of $G$. Once this has been established, 
we show how the constructions used in Section 3 extend to the general case.  In Section 5 we show how to modify the construction in
\cite{BCW} to obtain an infinite sequence of graphs $\{G_d\}_{d \ge 1} $ such that $G_d$ has 
non-vanishing homology in dimension $d$.
Section 6 concludes with general remarks about issues involved in generalizing the results in this paper to arbitrary graphs.

Since we will be discussing and comparing several different homology theories for graphs, it may be helpful to clarify our terminology in advance. Our primary focus is on the {\it discrete singular cubical homology} of a graph $G$, defined in Section 2). It will be denoted by 
$\Hom^\Cube(G)$ and occasionally called, simply, the
{\it discrete cubical homology}. If $X$ is any topological space, one can construct a {\it singular cubical homology} (e.g.,
\cite{Massey91}) and a {\it singular simplicial homology}
(e.g., \cite{Hat},\cite{Mu}). Since these are equal under assumptions relevant to this paper, they will both be called
the {\it (ordinary) singular homology} 
of $X$ and denoted by
$\Hom^\Sing(X)$. We will also have occasion to use the
fact that $\Hom^\Sing(X)\approx \Hom^\Cell(X)$, the 
{\it cellular homology} of $X$, when $X$ is a CW-complex.

\section{Background: discrete cubical homology of graphs}

We will briefly review the definitions relevant to this paper, referring the reader to \cite{BGJW} for more details and examples, and also to
 \cite{Diestel00} for graph theory definitions and terminology. For any positive integer $n$, let  
$[n] := \{1,\dots,n\}$. Throughout the paper, all homology computations will be done over a commutative ring with identity, denoted $R.$

For $d \geq 1$, the \textit{discrete $d$-cube $Q_d$} is the graph with vertex set $$V(Q_d)=\big\{(a_1,\dots, a_d)\,\big|\,a_i \in \{0,1\}, 1 \le i \le d\big\},$$  
and edge set $E(Q_d)$ consisting of those pairs of vertices $\{a,b\}$ differing in exactly one position. 
By convention,  $Q_0$ is the $1$-vertex graph with no edges.

If $G$ and $H$ are simple graphs, i.e. undirected graphs without loops or multiple edges, a \textit{graph homomorphism} (or \textit{graph map}) $\sigma: G \longrightarrow H$ is a 
map from $V(G)$ to $V(H)$ such that, if $\{a,b\}\in E(G)$ then either $\sigma(a)=\sigma(b)$ or $\{\sigma(a),\sigma(b)\}\in E(H)$.
A graph homomorphism $\sigma: Q_d \longrightarrow G$ is
called a \textit{singular} $d$-cube on $G$.

For each 
$d \geq 0$, let $\LL^\Cube_d(G)$ be the free $R$-module generated by all singular 
$d$-cubes on $G$.
For $d \geq 1$ and each $i \in [d]$, we define two face maps $f_i^+$ and $f_i^-$ from 
$\LL^\Cube_d(G)$ to $\LL^\Cube_{d-1}(G)$ such that, 
for $\sigma \in \LL^\Cube_d(G)$ and $(a_1,\dots,a_{d-1}) \in Q_{d-1}$:
\begin{eqnarray*}
  f_{i}^{+}\sigma(a_1,\dots,a_{d-1})&:=&\sigma(a_1,\dots,a_{i-1},1,a_i,\dots,a_{d-1}),  \\
  f_{i}^{-}\sigma(a_1,\dots,a_{d-1})&:=& \sigma(a_1,\dots,a_{i-1},0,a_i,\dots,a_{d-1}).
\end{eqnarray*}
For $d \geq 1$, a singular $d$-cube $\sigma$ is 
called \textit{degenerate} 
if $\sigma$ does not depend on at least one of its variables, that is, $f_{i}^{+}\sigma = f_{i}^{-}\sigma$, for some $i \in [d]$. Otherwise, 
$\sigma$ is called \textit{non-degenerate}.  By definition every
$0$-cube is non-degenerate.

For each $d\geq 0$, let $\DD^\Cube_d(G)$ be the 
$R$-submodule of 
$\LL^\Cube_d(G)$ generated by all degenerate singular $d$-cubes, 
and let $\CC^\Cube_d(G) = \LL^\Cube_d(G)/\DD^\Cube_d(G)$, 
whose elements are called $d$-chains.
Clearly, the cosets of non-degenerate $d$-cubes freely generate $\CC^\Cube_d(G)$.

Furthermore, for each $d\geq 1$, define the boundary operator 
\[
\partial_d^\Cube :\LL^\Cube_d(G) \longrightarrow \LL^\Cube_{d-1}(G)
\]
such that, for each singular $d$-cube $\sigma$,
\[
\partial_d^\Cube(\sigma) = \sum_{i=1}^d (-1)^i \big( f_i^-\sigma - f_i^+\sigma \big)
\]
and extend linearly to all chains in $\LL^\Cube_d(G)$.
When there is no danger of confusion, we will 
abbreviate $\partial^\Cube_d$ as $\partial_d$. 
If one sets $\LL^\Cube_{-1} (G) =\DD^\Cube_{-1} (G)= (0)$ then one can define
$\partial_0^\Cube$ as the trivial map from $\LL_0^\Cube(G)$ to
$\LL_{-1}^\Cube(G)$.

It is easy to check that, for $d\geq 0$,
$\partial_d [\DD^\Cube_d(G)] \subseteq \DD^\Cube_{d-1}(G)$ 
and  $\partial_{d} \partial_{d+1} \sigma= 0$ 
(see \cite{BCW}). 
Hence,  using the same notation, we may define a boundary operator
$\partial_d: \CC^\Cube_d(G) \longrightarrow \CC^\Cube_{d-1}(G)$,
and  $\CC^\Cube(G) = \{\CC_d^\Cube(G),\partial_d\}_{d\ge 0}$ is a 
chain complex of free $R$-modules. 

\begin{Definition}
  For $d \geq 0$, denote by $\Hom^\Cube_d(G)$ the $d$\textsuperscript{th} 
  homology group of the chain complex $\CC^\Cube(G)$. In other words, 
  $\Hom^\Cube_d(G) := \Ker\,\partial_d/\Image\,\partial_{d+1}$.
\end{Definition}

We denote singular $d$-cubes $\sigma: Q_d \to G$ by sequences of length $2^d$, where
the $i$\textsuperscript{th} term is the value of $\sigma$ on the $i$\textsuperscript{th} vertex, and the vertices of $Q_d$ are indexed in colexicographic order. For example, if 
$G$ is a path with vertices labeled $\{1,2,3\}$, then the sequence
$(1,2,2,1,2,3,3,2)$ represents the singular $3$-cube with labels as illustrated below.
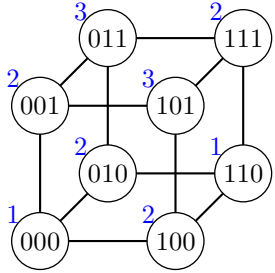
\begin{center}
\begin{tikzpicture}[scale=.9]
   \Vertex[x=0 ,y=0,L=$000$]{1}
      \Vertex[x=2 ,y=0,L=$100$]{2}
         \Vertex[x=1 ,y=1,L=$010$]{3}
            \Vertex[x=3 ,y=1,L=$110$]{4}
               \Vertex[x=0 ,y=2,L=$001$]{5}
                  \Vertex[x=2 ,y=2,L=$101$]{6}
                     \Vertex[x=1 ,y=3,L=$011$]{7}
                        \Vertex[x=3 ,y=3,L=$111$]{8}
      \Edge(1)(2)  \Edge(1)(3)  \Edge(2)(4)
      \Edge(3)(4)  \Edge(5)(6)\Edge(5)(7)\Edge(6)(8)\Edge(7)(8)
      \Edge(1)(5)\Edge(2)(6)\Edge(3)(7)\Edge(4)(8)
      \node [color=blue] at ([shift={(-.4,.4)}]1) {1};
         \node  [color=blue] at ([shift={(-.4,.4)}]2) {2};
            \node  [color=blue] at ([shift={(-.4,.4)}]3) {2};
               \node  [color=blue] at ([shift={(-.4,.4)}]4) {1};
                  \node  [color=blue] at ([shift={(-.4,.4)}]5) {2};
                     \node  [color=blue] at ([shift={(-.4,.4)}]6) {3};
                        \node  [color=blue] at ([shift={(-.4,.4)}]7) {3};
                           \node  [color=blue] at ([shift={(-.4,.4)}]8) {2};
\end{tikzpicture}
\end{center}
By convention, we will represent each coset in $\CC_d^\Cube(G)$ by the unique coset representative in which all terms are non-degenerate.

\section{Subdivision Map for the Pentagon}\label{section-pentagon}
In this section we prove that $\Hom^\Cube_d(G) = (0)$ for all $d \geq 2$ when $G=\ZZ_5$. The key step is to reduce the computation to
considering singular cubes $\sigma: Q_d \to \ZZ_5$ whose image has size at most $2$. 
\begin{Definition}For any graph $G$, 
\[
\CC^{(2)}_d(G) = \{ \sigma \in \CC_d(G) \;\big|\; |\textrm{Im}(\sigma)| \le 2 \}.
\]
\end{Definition}

\begin{Theorem}[Subdivision for $\ZZ_5$]\label{subdiv-pentagon} 
For all $d\ge 1$, there exists a map $S^d: \CC_d(\ZZ_5) \to \CC_d(\ZZ_5)$   such that

\begin{enumerate}
\item[(1)] $S^d$ is a chain map.
\item[(2)] $\textrm{Im}(S^d) \subseteq \CC^{(2)}_d(\ZZ_5) $.
\item[(3)] There exists maps $h_{d-1}: \CC_{d-1}(\ZZ_5) \to \CC_d(\ZZ_5)$ and
$h_d: \CC_d(\ZZ_5) \to \CC_{d+1}(\ZZ_5)$ such that 
\begin{enumerate}
\item[(i) ]for all $\sigma \in \CC_d(\ZZ_5)$, 
$
\sigma - S^d(\sigma) = h_{d-1} \partial_d(\sigma) + \partial_{d+1} h_d(\sigma), and
$
\item[(ii)] for all $\sigma \in \CC^{(2)}_d(\ZZ_5)$, $h_d(\sigma) \in \CC^{(2)}_{d+1}(\ZZ_5)$.
\end{enumerate}
\end{enumerate}
\end{Theorem}

Before proving \Ref{subdiv-pentagon} we will show that it implies one of our principal results.

\begin{Corollary}\label{corr-pentagon}
For all $d\ge 1$, $\Hom^\Cube_d(\CC(\ZZ_5)) =\Hom^\Cube_d(\CC^{(2)}(\ZZ_5))$.
\end{Corollary}

\begin{proof}
We may represent the maps defined in \Ref{subdiv-pentagon} by the following diagram.
\[
\xymatrix{
\cdots \ar[r]_>>>>>\partial & C_{d+1}  \ar[r]_>>>>>>{\partial} \ar[d]^{\color{blue}S} \ar@/_2.5pc/[dd]_>>>>>>{ \color{blue} I} & 
C_d \ar[r]_>>>>>{\partial}\ar[d]^{\color{blue}S}\ar@[red][ddl]_>>>>>>>>>>>>>>>>>>>{\color{red}h_d} \ar@/_2.5pc/[dd]_>>>>>>{\color{blue}I}&
C_{d-1} \ar[d]^{\color{blue}S}  \ar[r]_>>>>>\partial \ar@[red][ddl]_>>>>>>>>>>>>>>>>>>>{\color{red}h_{d-1}}\ar@/_2.5pc/[dd]_>>>>>>>{\color{blue}I}& \cdots\\
\cdots \ar[r]_>>>>>\partial & C^{(2)}_{d+1}  \ar[r]_>>>>>{\partial} \ar[d]^{\color{blue}i}  & 
C^{(2)}_d \ar[r]_>>>>{\partial}\ar[d]^{\color{blue}i} &
C^{(2)}_{d-1} \ar[d]^{\color{blue}i}  \ar[r]_>>>>>\partial & \cdots\\
\cdots \ar[r]_>>>>>\partial & C_{d+1}  \ar[r]_>>>>>{\partial} & 
C_d \ar[r]_>>>>>{\partial}  &
C_{d-1} \ar[r]_>>>>>\partial & \cdots
}
\] Property (3i) defines a chain homotopy between the identity map $I$ and $iS$ where $i$ denotes the inclusion map from 
$\CC^{(2)}$ to $\CC$. Since $S$ and $i$ are chain maps, they induce maps $S_*: H(\CC) \to H(\CC^{(2)})$ and $i_*: H(\CC^{(2)}) \to H(\CC)$ on homology, and chain homotopy implies  $I_* = (iS)_* = i_* S_*$. Since $S_*$ has a left inverse, it is injective. 

By property (3ii) we can also regard $h_d$ as a map from $\CC^{(2)}_d$ to $\CC^{(2)}_{d+1}$, and we obtain the following similar diagram.
\[
\xymatrix{
\cdots \ar[r]_>>>>>\partial & C^{(2)}_{d+1}  \ar[r]_>>>>>>{\partial} \ar[d]^{\color{blue}i} \ar@/_2.5pc/[dd]_>>>>>>{ \color{blue} I} & 
C^{(2)}_d \ar[r]_>>>>>{\partial}\ar[d]^{\color{blue}i}\ar@[red][ddl]_>>>>>>>>>>>>>>>>>>>{\color{red}h_d} \ar@/_2.5pc/[dd]_>>>>>>{\color{blue}I}&
C^{(2)}_{d-1} \ar[d]^{\color{blue}i}  \ar[r]_>>>>>\partial \ar@[red][ddl]_>>>>>>>>>>>>>>>>>>>{\color{red}h_{d-1}}\ar@/_2.5pc/[dd]_>>>>>>{\color{blue}I}& \cdots\\
\cdots \ar[r]_>>>>>\partial & C_{d+1}  \ar[r]_>>>>>{\partial} \ar[d]^{\color{blue}S}  & 
C_d \ar[r]_>>>>{\partial}\ar[d]^{\color{blue}S} &
C_{d-1} \ar[d]^{\color{blue}S}  \ar[r]_>>>>>\partial & \cdots\\
\cdots \ar[r]_>>>>>\partial & C^{(2)}_{d+1}  \ar[r]_>>>>>{\partial} & 
C^{(2)}_d \ar[r]_>>>>>{\partial}  &
C^{(2)}_{d-1} \ar[r]_>>>>>\partial & \cdots
}
\]
This describes a homotopy between $I: \CC^{(2)}\to \CC^{(2)}$ and $Si: \CC^{(2)} \to \CC^{(2)}$. Hence, for the induced maps on homology, we have $I_* = (Si)_* = S_* i_*$ which  
proves that $S_*$ has a right inverse. Hence $S_*$ is surjective, and we conclude that it is an isomorphism of homology. 

\end{proof}

\begin{proof}[Proof of \Ref{subdiv-pentagon}]
We will first construct, for any integer $N\ge 2$, an operator $S^N: \CC_d(\ZZ_5) \to \CC_d(\ZZ_5)$ which maps a singular cube $\sigma\in \CC_d(\ZZ_5) $ to a sum of $N^d$ singular cubes. Then, in order to ensure property (3i) of \Ref{subdiv-pentagon} we will specialize to $N=d$. 

The construction proceeds in several steps, which we will first illustrate by a small example with $N=3, d=2$. Consider the singular $2$-cube $\sigma = (1,2,2,3) \in \CC_2(\ZZ_5)$  illustrated by the following picture.
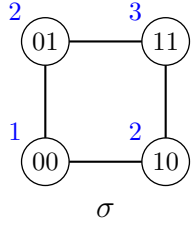
\begin{center}
\begin{tikzpicture}[scale=.8]
   \Vertex[x=0 ,y=0,L=$00$]{1}
      \Vertex[x=2 ,y=0,L=$10$]{2}
               \Vertex[x=0 ,y=2,L=$01$]{3}
                  \Vertex[x=2 ,y=2,L=$11$]{4}
      \Edge(1)(2)  \Edge(1)(3)  \Edge(2)(4)
      \Edge(3)(4)  
      \node [color=blue] at ([shift={(-.5,.5)}]1) {1};
         \node  [color=blue] at ([shift={(-.5,.5)}]2) {2};
            \node  [color=blue] at ([shift={(-.5,.5)}]3) {2};
               \node  [color=blue] at ([shift={(-.5,.5)}]4) {3};
               \node at (1,-.8) {\Large$\sigma$};
             \end{tikzpicture}
\end{center}
The steps in the construction of $S^3(\sigma)$ are as follows:

\begin{itemize}
\item[\sc Step 1:]
Lift $\sigma$ to a graph map $\tilsig: Q_2 \to \ZZ$. In this case, the above picture also 
represents $\tilsig$, but this may not be true in general.
\item[\sc Step 2:] Subdivide $Q_2$, creating a grid $Q_2^3$ with nine small subsquares. 
Extend $\tilsig$ to
a map $\tilsig^3: Q_2^3 \to \QQ$ by  computing weighted averages along lines.  
\item[\sc Step 3:] Round down, obtaining a map $\lfloor \tilsig^3 \rfloor: Q_2^3 \to \ZZ$, and then reduce mod $5$, obtaining a map $[\tilsig^3] : Q_2^3 \to \ZZ_5$. 
\end{itemize}
The results of steps 2 and 3, constructing $\tilsig^3$ and $[\tilsig^3]$, may be represented as follows:

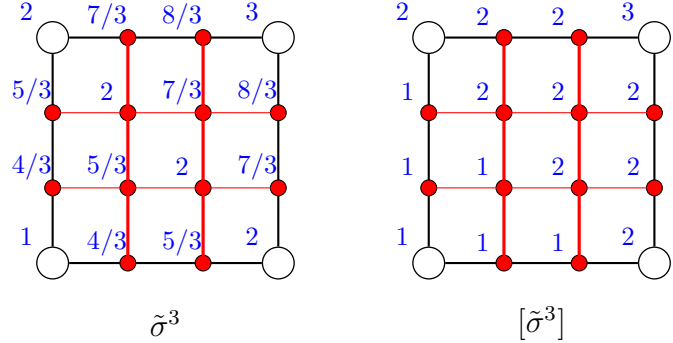
\begin{center}
\begin{tikzpicture}[scale=1]
\SetGraphUnit{2}\GraphInit[vstyle=Hasse]
\def\shfta{.35}
\def\shftb{.28}


\Vertex[x=0 ,y=0,L=$00$]{1}
\Vertex[x=3 ,y=0,L=$10$]{2}
\Vertex[x=0 ,y=3,L=$01$]{3}
\Vertex[x=3 ,y=3,L=$11$]{4}
\Edge(1)(2) \Edge(1)(3) \Edge(2)(4) \Edge(3)(4)
\node [color=blue] at ([shift={(-\shfta,\shfta)}]1) {1};
         \node  [color=blue] at ([shift={(-\shfta,\shfta)}]2) {2};
            \node  [color=blue] at ([shift={(-\shfta,\shfta)}]3) {2};
               \node  [color=blue] at ([shift={(-\shfta,\shfta)}]4) {3};
               
\Vertex[x=5 ,y=0,L=$00$]{1}
\Vertex[x=8 ,y=0,L=$10$]{2}
\Vertex[x=5 ,y=3,L=$01$]{3}
\Vertex[x=8 ,y=3,L=$11$]{4}
\Edge(1)(2) \Edge(1)(3) \Edge(2)(4) \Edge(3)(4)
\node [color=blue] at ([shift={(-\shfta,\shfta)}]1) {1};
         \node  [color=blue] at ([shift={(-\shfta,\shfta)}]2) {2};
            \node  [color=blue] at ([shift={(-\shfta,\shfta)}]3) {2};
               \node  [color=blue] at ([shift={(-\shfta,\shfta)}]4) {3};

\foreach \x in {1,2, 6,7}
\foreach \y in {0,1,2,3}
{
\draw[fill=red] (\x,\y) circle (3pt);
}

\foreach \x in {0,3,5,8}
\foreach \y in {1,2}
{
\draw[fill=red] (\x,\y) circle (3pt);
}

\foreach \x in {0,5}
\foreach \y in {1,2}
{
 \draw[red] (\x,\y) -- (\x+3,\y); 
};

\foreach \x in {1,2,6,7}
{
 \draw[very thick,red] (\x,0) -- (\x,3); 
};

\node  [color=blue] at (1-\shftb,0+\shftb) {4/3};
\node  [color=blue] at (2-\shftb,0+\shftb) {5/3};
\node  [color=blue] at (0-\shftb,1+\shftb) {4/3};
\node  [color=blue] at (0-\shftb,2+\shftb) {5/3};
\node  [color=blue] at (1-\shftb,3+\shftb) {7/3};
\node  [color=blue] at (2-\shftb,3+\shftb) {8/3};
\node  [color=blue] at (3-\shftb,1+\shftb) {7/3};
\node  [color=blue] at (3-\shftb,2+\shftb) {8/3};
\node  [color=blue] at (1-\shftb,1+\shftb) {5/3};
\node  [color=blue] at (2-\shftb,1+\shftb) {2};
\node  [color=blue] at (1-\shftb,2+\shftb) {2};
\node  [color=blue] at (2-\shftb,2+\shftb) {7/3};
\node  [color=blue] at (6-\shftb,0+\shftb) {1};
\node  [color=blue] at (7-\shftb,0+\shftb) {1};
\node  [color=blue] at (5-\shftb,1+\shftb) {1};
\node  [color=blue] at (5-\shftb,2+\shftb) {1};
\node  [color=blue] at (6-\shftb,3+\shftb) {2};
\node  [color=blue] at (7-\shftb,3+\shftb) {2};
\node  [color=blue] at (8-\shftb,1+\shftb) {2};
\node  [color=blue] at (8-\shftb,2+\shftb) {2};
\node  [color=blue] at (6-\shftb,1+\shftb) {1};
\node  [color=blue] at (7-\shftb,1+\shftb) {2};
\node  [color=blue] at (6-\shftb,2+\shftb) {2};
\node  [color=blue] at (7-\shftb,2+\shftb) {2};

\node at (1.5,-.8) {\Large$\tilsig^3$};
\node at (6.5,-.8) {\Large$[\tilsig^3]$};

\end{tikzpicture}
\end{center}

\begin{itemize}
\item[\sc Step 4:] Finally, define $S^3(\sigma)$ to be the sum of the (non-degenerate) small subcubes
appearing in $[\tilsig^3]$. That is,
\[
S^3(\sigma) = 2(1,1,1,2) +  3(1,2,2,2)  +(2,2,2,3).
\]
\end{itemize}
Next we explain the construction of $S^N$ in general, and show that it satisfies property (1) of \Ref{subdiv-pentagon}.

\begin{Definition}[Discrete $N$-grid in dimension $d$] For any $d, N \ge 1$, let
\[
Q_d^N = \{(a_1,\dots, a_d) \;|\; a_i \in \{0,\dots,N\}\,\}.
\]
\end{Definition}

\begin{Lemma}[Lifting lemma]\label{lift-pentagon}
If  $\sigma: Q_d \to \ZZ_5$ is a singular $d$-cube, there exists a graph map
$\tilsig: Q_d \to \ZZ$ such that the diagram
\[
\xymatrix{
Q_d  \ar[r]^>>>>>>{\tilsig} \ar@/_1.5pc/[rr]^>>>>>>>{ \sigma} & 
\ZZ \ar[r]_>>>>>{} & \ZZ_5
}
\]
commutes.  This map is unique up to translation of its image.
\end{Lemma}

\begin{proof}[Proof of \Ref{lift-pentagon}]
Choose $v_0 = (0,\dots,0)$ in $Q_d$ as a basepoint, and define $\tilsig(v_0)\in \ZZ$ to 
be the 
minimal positive representative of the residue class of 
$\sigma(v_0)\in \ZZ_5$. Define $\tilsig$ on 
all of $Q_d$ by extending along paths from $v_0$ by the following rules: if
$\tilsig(v) = k \in \ZZ$ has been defined, and $u$ is adjacent to $v$ in $Q_d$, then
\begin{equation}\label{rules}
\tilsig(u) = 
\begin{cases} 
k &\mbox{if } \sigma(u) = \sigma(v) \\
k+1&\mbox{if } \sigma(u) = \sigma(v)+ 1 \mod 5 \\
k-1 &\mbox{if } \sigma(u) = \sigma(v)-1 \mod 5 .\\
\end{cases}
\end{equation}
It is clear that if $\tilsig$ is well defined, then it is a graph map.  To show that it is well defined, it suffices to show that applying
\Ref{rules} iteratively defines $\tilsig$ unambiguously around any loop in $Q_d$. For completeness, we sketch a short proof of this fact, which is
probably well known.

Suppose $\gamma = (v_0, v_1, \dots, v_{2k} = v_0)$ is a loop of length $2k$ in $Q_d$. We may represent $\gamma$ as a loop in the Boolean lattice
${\bf 2}^d$, where $v_0$ corresponds to the empty set and steps consist of additions or deletions of elements in $[d]$.  If additions and deletions of elements $x\in [d]$ are encoded by $x$ and $\bar{x}$, respectively, then $\gamma$ may be represented by
a word in these symbols.  For example,  $w = 35\bar{3} 2 \bar{5} 3$ represents the path
\[
\emptyset \longrightarrow 3 \longrightarrow 35 \longrightarrow 5 \longrightarrow 25 \longrightarrow 2 \longrightarrow 23.
\]
A word $w$ represents a path in $Q_d$ if and only if for each $x \in [d]$, the occurrences of $x$ and $\bar{x}$ in $w$ form an alternating subsequence beginning with $x$. It represents a loop if an only if, for all $x$, the sequence also ends with $\bar{x}$. 
The property of representing a loop is preserved if we transpose symbols $w_i$ and $w_{i+1}$, where $w_i=x$ or $\bar{x}$, $w_{i+1} = y$ or $\bar{y}$, and $x \ne y$. Consequently, if $w$ represents a loop 
$\gamma(w)$ beginning and ending at $\emptyset$, we can transform $w$ 
by adjacent transpositions into a word of the form
$w^* = x \bar{x} y \bar{y}\cdots z \bar{z}$ with the same properties.  Further, if $\tilsig(w)$ is the result of applying $\tilsig$ to the loop $\gamma(w)$ in $Q_d$, it is straightforward to show (by examining a small number of cases) that transforming adjacent letters in $w$ does not do not change the endpoint of $\tilsig(w)$.  Since $\gamma(w^*)$ obviously maps to a closed loop in $\ZZ_5$,  the same must be true of $\gamma(w)$, and we are done. 
\end{proof}

Next we extend $\tilsig$ to the interior of $Q_d^N$, obtaining a map $\tilsig^N: Q_d^N \to \QQ$.

\begin{Definition}[Weighted $\QQ$-averaging] \label{ave-pentagon}Suppose $\sigma: Q_d \to \ZZ_5$ is a graph map, and $\tilsig: Q_d \to \ZZ$ extends $\sigma$ as described in   
\Ref{lift-pentagon}. Define $\tilsig^N: Q_d^N \to \QQ$ as follows: if $(a_1,\dots,a_d)$ is a vertex of $Q_d^N$, then 
\begin{equation}\label{grid-ext}
\tilsig^N(a_1,\dots,a_d) = 
\sum_{v = (\epsilon_1,\dots,\epsilon_d)} w_v(a_1,\dots,a_d) \tilsig(\epsilon_1,\dots,\epsilon_d),
\end{equation}
where the sum is over all vertices $v=(\epsilon_1,\dots,\epsilon_d)\in Q_d$, where $\epsilon_i \in \{0,1\}$, and where the weights $w_v(a_1,\dots,a_d)$ are
defined by 
\begin{equation}\label{weights-pentagon}
w_{(\epsilon_1,\dots,\epsilon_d)} = 
\prod_{i=1}^d
\begin{cases} 
1-\frac{a_i}{N} &\mbox{if } \epsilon_i = 0 \\
\frac{a_i}{N} &\mbox{if }  \epsilon_i = 1. \\
\end{cases}
\end{equation}
\end{Definition}

It is easy to see that for fixed $(a_1,\dots,a_d)$ we have 
\[
\sum_v w_v(a_1,\dots,a_d) = 1.
\]
Thus, $\tilsig^N(a_1,\dots,a_d)$ is a weighted average of $\tilsig(v)$, over all 
vertices $v$ of $Q_d$, taking values in $\QQ$. Furthermore, formula \Ref{ave-pentagon} 
is equivalent to a
recursive construction, where $\tilsig^N$ is first defined on edges of $Q_d$, then extending to higher dimensional
faces by computing weighted averages along lines between corresponding points on opposing faces (in any order).

\begin{Lemma}\label{ineq-pentagon}
If $\sigma: Q_d \to \ZZ_5$ is a graph map, then for any adjacent vertices $u,v$ of $Q_d^N$, we have
\begin{equation}
\big|\, \tilsig^N(u) - \tilsig^N(v)\,\big| \le 1/N.
\end{equation}
\end{Lemma}

\begin{proof}[Proof of \Ref{ineq-pentagon}]
If $d=1$ this is immediate. Suppose, inductively, that we have proved \Ref{ineq-pentagon} for 
maps with domain $Q_{d-1}^N$. Then it 
holds for adjacent vertices on the outer faces of $Q_d^N$, by the inductive construction of $\tilsig^N$. Suppose
that $(u,v)$ is an interior edge of $Q_d^N$. Draw lines perpendicular to $(u,v)$, meeting opposite faces of $Q_d^N$
at points $p,q,r$ and $s$, as shown:
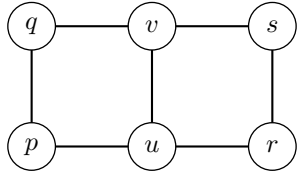
\begin{center}
\begin{tikzpicture}[scale=.8]
   \Vertex[x=0 ,y=0,L=$p$]{1}
      \Vertex[x=0 ,y=2,L=$q$]{2}
       \Vertex[x=2 ,y=0,L=$u$]{3}
         \Vertex[x=2 ,y=2,L=$v$]{4}
           \Vertex[x=4 ,y=0,L=$r$]{5}
         \Vertex[x=4 ,y=2,L=$s$]{6}
         \Edge(1)(2) \Edge(3)(4) \Edge(5)(6)
          \Edge(1)(3) \Edge(3)(5) \Edge(2)(4)\Edge(4)(6)
      \end{tikzpicture}
   \end{center}
By hypothesis, both $\tilsig^N(p) - \tilsig^N(q) $ and $\tilsig^N(r) - \tilsig^N(s) $ lie
in the interval $[-1/N,1/N]$, and we also have
\begin{eqnarray*}
\tilsig^N(u) &=& \alpha \tilsig^N(p) + (1-\alpha) \tilsig^N(r), \textrm{ and} \\
\tilsig^N(v) &=& \alpha \tilsig^N(q) + (1-\alpha) \tilsig^N(s) 
\end{eqnarray*}
for some $\alpha \in [0,1]$. It follows that 
$\tilsig^N(u) - \tilsig^N(v) $ is a convex combination of
$\tilsig^N(p) - \tilsig^N(q) $ and $\tilsig^N(r) - \tilsig^N(s) $, 
and hence lies in $[-1/N,1/N]$ as claimed.
\end{proof}

\begin{Definition}\label{round-pentagon}
If $\sigma: Q_d\to \ZZ_5$ is a graph map, then 
$\lfloor \tilsig^N \rfloor : Q_d^N \to \ZZ$ is the map obtained by rounding down $\tilsig^N$ at each
vertex of $Q_d^N$. That is, 
\begin{equation}\label{roundmap}
\lfloor \tilsig^N \rfloor (a_1,\dots,a_d) = \lfloor \tilsig^N(a_1,\dots,a_d)  \rfloor.
\end{equation}
Finally, define
\begin{equation}\label{resmap}
[ \tilsig^N]  (a_1,\dots,a_d) = \lfloor \tilsig^N \rfloor (a_1,\dots,a_d) \mod 5 \in \ZZ_5.
\end{equation}
\end{Definition}

\begin{Corollary}\label{keysteps-pentagon}
Suppose that $\sigma: Q_d \to \ZZ_5$ is a graph map, and 
$[\tilsig^N]: Q_d^N \to \ZZ_5$ is defined
as in \Ref{resmap}. Then
\begin{enumerate}
\item $[\tilsig^N]: Q_d^N \to \ZZ_5$ is a graph map.
\item If $N \ge d$ and $u,v$ are any vertices of a small subcube in $Q_d^N$, then\newline
$\big| [\tilsig^N](u) - [\tilsig^N](v)\big| \le 1$. 
\end{enumerate}
In particular, the image of $\tilsig^N$ restricted to
each of the small subcubes in $Q_d^N$ has size at most two.
\end{Corollary}
\begin{proof}[Proof of \Ref{keysteps-pentagon}]
Statement (1) follows immediately from \Ref{ineq-pentagon} and the elementary fact that if $\alpha, \beta\in \RR$ and 
$|\alpha-\beta|\le 1$, then $\big| \lfloor \alpha \rfloor - \lfloor \beta \rfloor \big| \le 1$. Statement (2) follows
from \Ref{ineq-pentagon} and the fact that any two vertices in $Q_d$ can be linked by a path of length $\le d$.
\end{proof}

Now we are finally in a position to define $S^N$, by summing the restrictions of
$[\tilde{\sigma}^N]$ to each of the small subcubes of $Q_d^N$.

\begin{Definition}\label{sn-pentagon}
For any $d\ge 1, N \ge 2$, define $S^N: \CC_d(\ZZ_5) \to \CC_d(\ZZ_5)$ by setting
\[
S^N(\sigma) = \sum_{q_d \in Q_d^N} [\tilsig^N] \,\big| q_d,
\]
where $\sigma$ a generator of $\CC_d(\ZZ_5)$,
and the sum is over all of the $N^d$ small subcubes $q_d$ of $Q_d^N$.
\end{Definition}

The next corollary includes the first two parts of \Ref{subdiv-pentagon}. 
\begin{Corollary}\label{cor-parts12-pentagon}
\begin{enumerate}
\item
For all $d\ge 1, N \ge 2$, the map $S^N$ is a chain map.
\item
If  $\sigma$ is a generator of $\CC_d(\ZZ_5)$, then each summand of $S^d(\sigma)$ is a graph map with image
of size at most $2$.
\end{enumerate}
\end{Corollary}
\begin{proof}[Proof of \Ref{cor-parts12-pentagon}]
The essential step in proving (1) is to show that 
the definition of $[\tilsig_N]$ on a face
$f_i^\epsilon Q_d^N$ of $Q_d^N$ agrees exactly
with the definition of 
$[\widetilde{(f_i^\epsilon \sigma)}_N]$ on $Q_{d-1}^N$. 
We leave the verification of this fact to the reader. 
Internal cancellation of faces of small cubes then implies identity $\partial_d S^N(\sigma) = S^N (\partial_d \sigma)$.
Statement (2) is a consequence of \Ref{keysteps-pentagon}.
\end{proof}
It remains to prove part (3) of \Ref{subdiv-pentagon}.  We will define maps $h_{d-1}$ and $h_d$ and show that
for all $\sigma\in \CC_d(\ZZ_5)$, 
\begin{equation}\label{chop-proof}
\partial_{d+1} h_d(\sigma) = \sigma - S^d(\sigma) - h_{d-1} \partial_d(\sigma) \pm
\textrm{ degenerate terms}.
\end{equation}
The argument generally follows \cite[\S 7.7]{Massey91}, where a similar construction is used to obtain an analogous result for ordinary cubical simplicial homology. However, there are significant variations, including the need for subdivisions $S^N$ with $N>2$ and also for more than two levels in the definition of $h_d$ and $h_{d-1}$. 
\begin{Definition}\label{hdef-pentagon}
For $\sigma\in \CC_d(\ZZ_5)$, we define $h_d$  by the following process.
\begin{enumerate}
\item
First construct a $\QQ$-labeling $\tilde{\gamma}(\sigma)$ of the $(d+1)$-dimensional grid $Q_{d+1}^d = \{(a_1,\dots, a_{d+1}) \;|\; 0 \le a_i \le d\}$ as follows:
\begin{enumerate}
\item
The bottom face ($a_{d+1}=0$) is $\tilsig^d$ (the second step in constructing $S^d(\sigma)$).
\item
The top face ($a_{d+1}= d$) is $T_d(\sigma)$, defined by the rule
\[
T_d(\sigma) (a_1,\dots,a_d) = \tilsig(\bar{a_1},\dots,\bar{a_d}),\;\;
\bar{a_i} = \begin{cases} 0 & \textrm{ if $a_i=0$}\\1 & \textrm{ if $a_i > 1$}.\end{cases}
\]
\item
Along each line from $(a_1,\dots,a_d,0)$ to $(a_1,\dots,a_d,d)$, assign labels by computing equally
spaced averages between $\tilde{\gamma}(\sigma)(a_1,\dots,a_d,0)$ and \newline$\tilde{\gamma}(\sigma)(a_1,\dots,a_d,d)$.
\end{enumerate}
\item
Transform $\tilde{\gamma}(\sigma)$ into a $\ZZ$-labeling $\lfloor\tilde{\gamma}(\sigma)\rfloor$ by rounding down, and then
into a $\ZZ_5$ labeling $[\tilde{\gamma}(\sigma)]$ by reducing mod $5$.
\item
Define $h_d(\sigma)$ to be $(-1)^{d+1}$ times the sum of the $d^d$ small subcubes in $[\tilde{\gamma}(\sigma)]$.
\end{enumerate}
\end{Definition}

\noindent
This construction is designed to embody a proof of \Ref{chop-proof} for a single value of $d$. For that purpose, $h_{d-1}$ must be 
defined compatibly with $h_d$, and not independently. More precisely, $h_{d-1}$ is defined 
as in \Ref{hdef-pentagon}, but on a grid of type $Q_d^d$, not $Q_{d-1}^d$. It will follow that the side faces of 
$h_d(\sigma)$ appear as images of $h_{d-1}$ applied to the faces of $\sigma$. 

The following partially labeled diagram illustrates the construction of $h_2: \CC_2(G) \to \CC_3(G)$ for a generic singular $2$-cube $\sigma = (a,b,c,d)$.

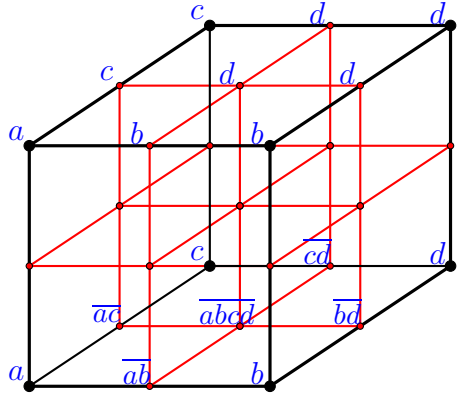
\begin{center}
\begin{tikzpicture}[scale=.4]
\def\shfta{.43}
\def\shftb{.5}


\draw[thick,red] (0,4) -- (8,4) -- (14,8) -- (6,8) -- (0,4);
\draw[thick,red] (4,0) -- (10,4) -- (10,12) -- (4,8) -- (4,0);
\draw[thick,red] (3,2) -- (11,2) -- (11,10) -- (3,10) -- (3,2);
\draw[thick,red] (7,2) -- (7,10); \draw[thick,red] (3,6) -- (11,6); \draw[thick,red] (4,4) -- (10,8);

\draw[very thick] (0,0) -- (8,0) -- (14,4) -- (14,12) -- (6,12) -- (0,8) -- (0,0) ; 
\draw[very thick] (0,8) -- (8,8) -- (14,12); \draw[very thick] (8,0) -- (8,8);
\draw[thick] (0,0) -- (6,4) -- (14,4); \draw[thick] (6,4) -- (6,12);

\foreach \x in {0,1,2}
\foreach \y in {0,1,2}
{\draw[fill=red] (4+3*\x,2*\x + 4*\y) circle (3pt);}

\foreach \x in {3,7,11}
\foreach \y in {2,6,10}
{\draw[fill=red] (\x,\y) circle (3pt);}

\foreach \x in {0,1,2}
\foreach \y in {0,1,2}
{\draw[fill=red] (4*\x+3*\y,4+2*\y) circle (3pt);}

\foreach \x in {0,1}
\foreach \y in {0,1}
\foreach \z in {0,1}
{\draw[fill=black] (8*\x+6*\z,8*\y+4*\z) circle (5pt);}

\node [color=blue] at (0-\shfta,0+\shfta) {\Large $a$};
\node [color=blue] at (8-\shfta,0+\shfta) {\Large $b$};
\node [color=blue] at (6-\shfta,4+\shfta) {\Large $c$};
\node [color=blue] at (14-\shfta,4+\shfta) {\Large $d$};
\node [color=blue] at (0-\shfta,8+\shfta) {\Large $a$};
\node [color=blue] at (8-\shfta,8+\shfta) {\Large $b$};
\node [color=blue] at (6-\shfta,12+\shfta) {\Large $c$};
\node [color=blue] at (14-\shfta,12+\shfta) {\Large $d$};

\node [color=blue] at (4-\shfta,8+\shfta) {\Large $b$};
\node [color=blue] at (3-\shfta,10+\shfta) {\Large $c$};
\node [color=blue] at (7-\shfta,10+\shfta) {\Large $d$};
\node [color=blue] at (11-\shfta,10+\shfta) {\Large $d$};
\node [color=blue] at (10-\shfta,12+\shfta) {\Large $d$};

\node [color=blue] at (4-\shfta,0+\shftb) {\large $\overline{ab}$};
\node [color=blue] at (3-\shfta,2+\shfta) {\large $\overline{ac}$};
\node [color=blue] at (7-\shfta,2+\shftb) {\large $\overline{abcd}$};
\node [color=blue] at (11-\shfta,2+\shftb) {\large $\overline{bd}$};
\node [color=blue] at (10-\shfta,4+\shftb) {\large $\overline{cd}$};


\end{tikzpicture}
\end{center}

\medskip
\noindent
According to 
\Ref{hdef-pentagon}, the vertex at the centroid of the grid should be labeled
\[
\frac{1}{2} \big(d + \frac{a+b+c+d}{4}\big) = \frac{a+b+c+5d}{8}.
\]

\smallskip

\begin{proof}[Proof of \Ref{subdiv-pentagon}, Part (3i)] The main difficulty is showing that
the labeling $[\tilde{\gamma}]$ defines a graph map from $Q_{d+1}^d$ to $\ZZ_5$. Once this is done, verifying statement {\em(3i)} reduces to a purely formal calculation.  

\begin{Lemma}\label{key-step-pentagon} If $\sigma \in \CC_d(\ZZ_5)$, then in the $\QQ$-labeling $\tilde{\gamma}(\sigma)$, every pair of adjacent labels differs by at most $1$.
\end{Lemma}
\begin{proof}[Proof of \Ref{key-step-pentagon}]
In the original $\sigma$, every pair of labels (adjacent or not) differs by at most $d$. This property carries over
to $\tilsig^d$, and every new edge in $\tilde{\gamma}(\sigma)$ is obtained by subdividing a segment from one
of the original labels to a label of $\tilsig^d$ into $d$ equal pieces. Hence the lengths of these pieces
are at most $d \times (1/d) = 1$.
\end{proof}
As we have observed in \Ref{keysteps-pentagon}, rounding preserves the 
property that labels of adjacent vertices differ by at most $1$, and reducing mod $5$ 
translates this into adjacency in $\ZZ_5$. It follows that $[\tilde{\gamma}(\sigma)]$ defines a graph map from $Q_{d+1}^d$ to $\ZZ_5$.

The proof of statement \Ref{chop-proof} is now straightforward. When the top face ($a_{d+1}=d$) is expanded into small cubes, it consists of $\sigma$ plus degenerate terms.  When the bottom face ($a_{d+1}=0$) is expanded, it is exactly $S^d(\sigma)$. Furthermore, the sign $(-1)^{d+1}$ of $h_d(\sigma)$ has been chosen 
exactly so that the terms in $\sigma - S^d(\sigma)$ appear with correct signs in the expansion of
$\partial_{d+1} h_d(\sigma)$. It is easy to see that the remaining terms (aside from $\sigma$) coming from the 
top face are all degenerate. Finally, $h_{d-1}$ has been defined exactly so that terms in the expansion of
$\partial_{d+1} h_d(\sigma)$ arising from faces other than $a_{d+1}=0$ and $a_{d+1} =d$ are are
accounted for (with proper sign) by $h_{d-1} \partial_d (\sigma)$.  
This completes the proof of part {\it (3i)} of 
\Ref{subdiv-pentagon}.
\end{proof} 

\begin{proof}[Proof of \Ref{subdiv-pentagon}, Part (3ii)]
We must show that if $\sigma$ has two labels, then 
every non-degenerate cube in the expansion of $h_d(\sigma)$ has two labels. Suppose $\sigma$ has two labels, and that these are mapped to $t$ and $t+1$ by the lifting of $\sigma$ to $\tilsig$.  Then the top face $T_d(\sigma)$ has these same
labels, and the bottom face $\tilsig^d$ has rational labels in the interval $[t,t+1]$. By construction, the remaining labels of $h_d(\sigma)$ 
all lie in the interval $[t,t+1]$. When these are rounded down,  the results will again all be equal to $t$ or $t+1$, and the result follows.
\end{proof} 
This completes the proof of \Ref{subdiv-pentagon}.
\end{proof} 

\Ref{subdiv-pentagon} shows that computing the homology of $\ZZ_5$ can be reduced to computing the homology of the 
chain complex $\CC^{(2)}(\ZZ_5)$, i.e. the complex generated by singular cubes whose image has size $\le 2$. The second main
result of this section shows that the complex 
$\CC^{(2)}(\ZZ_5)$ has trivial homology in all dimensions $d\ge 2$.

\begin{Theorem}\label{main-pentagon}
For all $d\ge 2$, $\Hom_{d}(\CC^{(2)}(\ZZ_5)) = (0)$, and consequently $\Hom_{d}(\CC(\ZZ_5)) = (0)$.
\end{Theorem}

\begin{proof}[Proof of \Ref{main-pentagon}]
Let $e = (i,i+1)$ (mod $5$) be an edge of $\ZZ_5$. For $d\ge 0$, define
\[
\CC_d^{(e)} (\ZZ_5)= \{ \sigma \in \CC_d(\ZZ_5) \;|\; \textrm{Im}(\sigma) \subseteq e\}.
\]
Clearly, for $d\ge 1$ we have $\partial[\CC_d^{(e)} (\ZZ_5)] \subseteq \CC_{d-1}^{(e)} (\ZZ_5)$. Furthermore, for
$d\ge 1$ we have 
$
\CC_d(\ZZ_5) = \bigoplus_e \CC_d^{(e)}(\ZZ_5).
$
It follows that for $d\ge 2$, $\Hom_{d}(\CC(\ZZ_5))\approx \bigoplus_e \Hom_{d}(\CC^{(e)}(\ZZ_5))$.
But $\Hom_{d}(\CC^{(e)}(\ZZ_5))$ is the $d$\textsuperscript{th}  homology of a contractible graph , and hence is trivial for all $d>0$
(see \cite{BGJW} for 
precise definitions, and for a proof of this result).
\end{proof}

\section{Graphs without $3$-Cycles or $4$-Cycles}

In this section we will extend \Ref{subdiv-pentagon}, \Ref{corr-pentagon}, and \Ref{main-pentagon} to arbitrary graphs $G$  
containing no $3$-cycles or $4$-cycles.  

\begin{Theorem}\label{thm-largegirth}Suppose that $G$ is a graph with
no $3$-cycles or $4$-cycles.
Then,
for $d\ge 1$ there exists a map $S^d: \CC_d(G) \to \CC_d(G)$ such that 
\begin{enumerate}
    \item $S^d$ is a chain map,
    \item
$\textrm{Im}(S^d) \subseteq \CC_d^{(2)}(G)$, and 
\item
for $d\ge 1$, 
$S^d$ induces an isomorphism of homology $S^d_*: \Hom_{d}(\CC(G)) \to  \Hom_{d}(\CC^{(2)}(G))$.
\end{enumerate}
Furthermore, for $d\ge 2$, $\Hom_{d}(\CC^{(2)}(G))=(0)$.
\end{Theorem}

\begin{proof}The proof will follow the steps in Section 3, showing that with small adjustments, all of the 
arguments generalize.  In fact, the only significant challenge is generalizing  \Ref{lift-pentagon}, proving
that singular cubes $\sigma$ with $\textrm{Im}(\sigma) = G$ can be lifted to a setting where
averages can be computed. This will require constructing the 
{\em universal covering graph} $U(G)$ of a graph $G$,  the discrete analog of  a familiar topological object
(see, e.g. \cite{Mu},\cite{Hat}). 

\begin{Definition}[Graph coverings]\label{covering} Suppose that $G$ is a connected (undirected) graph. 
\begin{enumerate} 
\item A  {\em covering} of $G$ by a 
graph $\tilde{G}$ is a map $p: \tilde{G}\to G$ such that for all
$\tilde{x} \in \tilde{G}$, $p$ maps the star $E_{\tilde{x}}$ of $\tilde{x}$ (the set of edges adjacent to $\tilde{x}$) bijectively onto the star $E_{p(\tilde{x})}$ of $p(\tilde{x})$.
\item
A {\em universal covering graph} $U(G)$ of $G$ is a graph with a covering map $u:U(G)\to G$, such that for any covering $p: \tilde{G} \to G$   there exists a covering map $v: U(G) \to \tilde{G}$ such that $u = pv$. 
\end{enumerate}
\end{Definition}

Much of the theory of topological coverings carries over to graphs, considered as $1$-dimensional
cell complexes. Universal covering graphs seem to have been first constructed explicitly
in \cite{Angluin}; see also, e.g., \cite{Leighton}, \cite{Sunada} for other applications.  For our purposes, we need the properties of $U(G)$ stated in the following lemma. 

\begin{Lemma}\label{univ-cover}
Suppose that $G$ is a finite connected graph. Then 
\begin{enumerate}
\item
$U(G)$ exists and is unique up to isomorphism. 
\item
If $G$ is a tree, then $U(G)=G$; otherwise $U(G)$ is an infinite tree. 
\item
Suppose that $G$ is a graph without $3$-cycles and $4$-cycles, and $\sigma: Q_d \to G$ is a graph map for some $d \ge 1$. Then there exists a graph map 
$\tilsig: Q_d \to U(G)$ such that $\sigma = u\tilsig$. 
\item For any $q\in Q_d$, the map $\tilsig$ in (3) may be chosen so that $\tilsig(q) = \tilde{q}$, where $\tilde{q}$ is any element of $u^{-1}(\sigma(q))$, and $\tilsig$ is uniquely determined by
that choice.
\end{enumerate}
\end{Lemma}

\begin{proof}[Sketch of Proof.]
$U(G)$ may be constructed by fixing a basepoint $v_0$ in $G$, and then defining the vertex set of $U(G)$ to be the set of paths
$(v_0,v_1,\dots,v_k)$ in $G$ that are ``non-backtracking", i.e., $v_{i-1} \ne v_{i+1}$ for $i=1,\dots,k-1$. Define the edges of $U(G)$ to be pairs of paths of the form $((v_0,v_1,\dots,v_k), (v_0,v_1,\dots,v_k, v_{k+1}))$. The covering map $u: U(G)\to G$ sends $(v_0,v_1,\dots,v_k)$ to $v_k$. The vertex $\tilde{v}_0 = (v_0)\in U(G)$ will be called the {\em root} of $U(G)$. It is clear that $U(G)$ is a tree, and (2) is immediate.

The proof of (3) is as in \Ref{lift-pentagon}, except that instead of explicit rules to extend $\tilsig$ along paths in $Q_d$, we use ``stars" to
guide the construction. More precisely, if $\tilsig(v) = \tilde{x}\in U(G)$ has been defined, and $u$ is adjacent to 
$v$ in $Q_d$, define $\tilsig(u) = \tilde{x}$ if $\sigma(u) = \sigma(v)$, and otherwise, 
if $\sigma(u) = v' \in E_{\sigma(v)} $, define $\tilsig(u)$ to be the unique vertex in 
$E_{\tilde{x}}$ whose image under $p$ equals $v'$. The argument that $\tilsig$ is
well defined is exactly the same as in \Ref{lift-pentagon}. Verification of (4) is left
to the reader.
\end{proof}

A more general version of  \Ref{univ-cover}(3) for arbitrary covering graphs
$\tilde{G}$ appears in \cite{Hard}. However, in 
the present paper we require only the special case
where $\tilde{G}=U(G)$, and the argument here is self-contained.

We proceed with the sequence of lemmas need to prove \Ref{thm-largegirth}, always assuming that $G$
is a finite connected graph without $3$-cycles or $4$-cycles.
In the remainder of the proof, we will assume that
$\sigma$ is a singular $d$-cube, and 
$\tilsig$ has been (uniquely) constructed so that if
$q_0 = (0,0,\dots,0) \in Q_d$ and $\sigma(q_0) = x_0 \in G$, then 
$\tilsig(q_0) = \tilde{x}_0\in u^{-1}(x_0)$ has 
been chosen so that $\tilde{x_0}$ is at least distance $d+2$ from the root
$\tilde{v}_0$ of $U(G)$.

By construction, $U(G)$ is a discrete infinite tree. We can embed $U(G)$ into a
``continuous" tree $\overline{U(G)}$, obtained by identifying each edge in $U(G)$ with a unit interval
in $\RR$. Next define a metric $d$ on $\overline{U(G)}$ by setting $d(x,y)$ equal to 
the length (under the induced metric) of the unique path from $x$ to $y$, for all $x,y\in \overline{U(G)}$.
This will allow us to compute {\em weighted averages}
$
\alpha x + (1-\alpha) y,
$
for any $x,y \in \overline{U(G)}$ and $\alpha \in [0,1]$, by identifying the path
from $x$ to $y$ in $\overline{U(G)}$ with a segment in $\RR$.
More precisely: 
\begin{Definition}\label{weight-ave}
Suppose that $x,y\in \overline{U(G)}$ and
$\alpha \in [0,1]$. Define $\alpha x + (1-\alpha) y$ to be the unique point 
$p$ in the path from $x$ to $y$ such that $d(x,p) = \alpha d(x,y)$.
\end{Definition}

Given $\sigma: Q_d \to G$ and its extension $\tilsig: Q_d \to U(G)$, we will define an extension
$\tilsig^N$ of $\tilsig$ to the grid $Q_d^N = \{(a_1,\dots,a_d) \;|\; 0 \le a_i \le N\}$, taking values
in $\overline{U(G)}$. The construction is similar to the one in Section 3, using weighted averages along paths in $\overline{U(G)}$. However, verifying that $\tilsig^N$ 
is a graph map requires extra care because the computation is not necessarily being carried out within a fixed interval in $\RR$.

\begin{Definition}[Definition of $\tilsig^N$]
Given $\sigma: Q_d \to G$ and its lift $\tilsig: Q_d \to U(G)$, define
$\tilsig^N: Q_d^N \to \overline{U(G)}$ as follows.
\begin{enumerate}
\item For each edge in $Q_d$ parallel to the first coordinate axis, i.e., from 
$(0,\epsilon_2,\dots,\epsilon_d)$ to $(1,\epsilon_2,\dots,\epsilon_d)$, for some sequence of $\epsilon_i\in \{0,1\}$, 
subdivide the path in $\overline{U(G)}$ from $\tilsig(0,\epsilon_2,\dots,\epsilon_d)$ to
$\tilsig(1,\epsilon_2,\dots,\epsilon_d)$ into $N$ pieces with $N-1$ equally
spaced points. Map the points $(k,N\epsilon_2,\dots,N\epsilon_d) \in 
\tilsig^N, 1 \le k \le N-1$, to these $N-1$ points in $\overline{U(G)}$. 
\item
Inductively, assume that values of $\tilsig^N$ have been assigned to all  $(i-1)$-faces of $Q_d^N$ that are parallel to the first $i-1$ coordinate vectors,
i.e., faces of the form $\{(a_1,\dots,a_{i-1},N\epsilon_i,\dots,N\epsilon_d) \}$,
where $\epsilon_i,\dots,\epsilon_d$ are fixed elements of $\{0,1\}$ and 
the $a_i$ range over all possible values in $\{0,\dots,N\}$. For each sequence
$a_1,\dots,a_{i-1}$, extend $\tilsig^N$ to points in the interior of the line from 
$(a_1,\dots,a_{i-1},0,\dots,0)$ to $(a_1,\dots,a_{i-1},N,\dots,N)$ by
subdividing the segment in $\overline{U(G)}$ from 
$\tilsig(a_1,\dots,a_{i-1},0,\dots,0)$ to $\tilsig(a_1,\dots,a_{i-1},N,\dots,N)$ into $N$ equal pieces, and assigning values of $\tilsig$ to interior points accordingly. This step assigns values of $\tilsig^N$ to all $i$-faces 
of $Q_d^N$.
\item
Repeat step (2) until $i=N$.
\end{enumerate}
\end{Definition}

In order to verify that $\tilsig^N$ has the desired properties, we will need the
following result, which is a generalization of \Ref{ineq-pentagon}.

\begin{Lemma} \label{ineq-gen}
Suppose that $x,y,u,v\in \overline{U(G)}$, and suppose $d(x,y)\le \delta$ and $ d(u,v) \le \delta$, where $\delta\le 1$. Then
for any $\alpha, 0 \le \alpha \le 1$, we have $d(\alpha x + (1-\alpha) u, \alpha y + (1-\alpha) v) \le \delta$.
\end{Lemma}
\begin{proof}
If $x,y,u$ and $v$ all lie on a single path in $\overline{U(G)}$, we can identify this path with an interval in $\RR$ and, for example,
the point $\alpha x + (1-\alpha) y$ may be computed using
ordinary real arithmetic. In this interval we have $x-y\in [-\delta,\delta]$ and $u-v \in [-\delta,\delta]$, which imply
\begin{equation}\label{ineq-proof}
(\alpha x + (1-\alpha) u)-(\alpha y + (1-\alpha) v) = \alpha(x-y) + (1-\alpha) (u-v) \in [-\delta, \delta],
\end{equation}
as desired. If $x,y,u$ and $v$ do not lie on a single path, then some 
additional argument is needed.

Let $p = \alpha x + (1-\alpha) u$ and $q = \alpha y + (1-\alpha)v$, where
these points are computed in $\overline{U(G)}$ as defined 
as in \Ref{weight-ave}, using
the metric on $\overline{U(G)}$.

Since $d(x,y)\le 1$, $x$ and $y$ must either lie on a single edge of $\overline{U(G)}$, or on two adjacent edges, and $u$ and $v$ are situated similarly. Up to obvious permutations of the labels, there are only two configurations representing the possible arrangements of $x,y,u$ and $v$ in
$\overline{U(G)}$, as shown in the following diagrams.

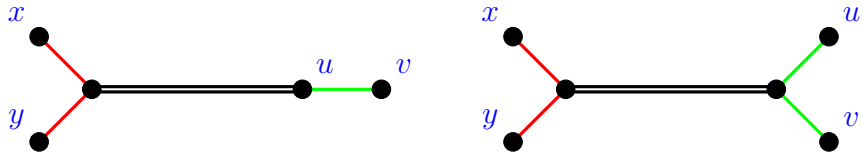
\begin{center}
\begin{tikzpicture}[scale=.7]
\def\shfta{.43}
\def\shftb{.5}


\draw[very thick] (1,.95) -- (5,.95) ; 
\draw[very thick] (1,1.05) -- (5,1.05) ; 
\draw[very thick,red] (0,2) -- (1,1) -- (0,0); 
\draw[very thick,green] (5,1) -- (6.5,1); 

\draw[very thick] (10,.95) -- (14,.95); 
\draw[very thick] (10,1.05) -- (14,1.05); 
\draw[very thick,red] (9,2) -- (10,1); 
\draw[very thick,red] (9,0) -- (10,1); 
\draw[very thick,green] (14,1) -- (15,2); 
\draw[very thick,green] (14,1) -- (15,0); 

\draw[fill=black] (9,0) circle (5pt);
\draw[fill=black] (9,2) circle (5pt);
\draw[fill=black] (10,1) circle (5pt);
\draw[fill=black] (14,1) circle (5pt);
\draw[fill=black] (15,2) circle (5pt);
\draw[fill=black] (15,0) circle (5pt);

\draw[fill=black] (0,0) circle (5pt);
\draw[fill=black] (1,1) circle (5pt);
\draw[fill=black] (0,2) circle (5pt);
\draw[fill=black] (5,1) circle (5pt);
\draw[fill=black] (6.5,1) circle (5pt);

\node [color=blue] at (0-\shfta,0+\shfta) {\Large $y$};
\node [color=blue] at (0-\shfta,2+\shfta) {\Large $x$};
\node [color=blue] at (5+\shfta,1+\shfta) {\Large $u$};
\node [color=blue] at (6.5+\shfta,1+\shfta) {\Large $v$};

\node [color=blue] at (9-\shfta,0+\shfta) {\Large $y$};
\node [color=blue] at (9-\shfta,2+\shfta) {\Large $x$};
\node [color=blue] at (15+\shfta,2+\shfta) {\Large $u$};
\node [color=blue] at (15+\shfta,0+\shfta) {\Large $v$};

\end{tikzpicture}
\end{center}
In both diagrams, the doubled lines represent the portions of the $xu$ and $yv$ paths that overlap, and the red and green segments represent portions that do not overlap. The overlapping portion must be non-empty, but 
could consist of a single point.

Let $T=T(x,y,u,v)$ denote the (metric) subtree of $\overline{U(G)}$ obtained by taking the union of the paths from $x$ to $u$ and $y$ to $v$. 
Let $\overline{T}\subseteq \RR$ denote an interval in $\RR$ obtained by 
taking the union of two intervals obtained by mapping each of the $xu$ and $yv$ paths isometrically into $R$, in such a way that the overlapping portions
coincide. For example, $\overline{T}$ might look like this, with $p$
and $q$ included:

\begin{center}
\begin{tikzpicture}[scale=.8]
\def\shfta{.3}


\draw[very thick] (1,.95) -- (5,.95) ; 
\draw[very thick] (1,1.05) -- (5,1.05) ; 
\draw[very thick,red] (-1,1) -- (1,1); 
\draw[very thick,green] (5,1) -- (6.5,1);

\draw[fill=black] (-1,1) circle (5pt);
\draw[fill=black] (0,1) circle (5pt);
\draw[fill=black] (1,1) circle (5pt);
\draw[fill=black] (2.5,1) circle (5pt);
\draw[fill=black] (3.5,1) circle (5pt);
\draw[fill=black] (5,1) circle (5pt);
\draw[fill=black] (6.5,1) circle (5pt);

\node [color=blue] at (-1-\shfta,1+\shfta) {\Large $y$};
\node [color=blue] at (0-\shfta,1+\shfta) {\Large $x$};
\node [color=blue] at (5+\shfta,1+\shfta) {\Large $u$};
\node [color=blue] at (6.5+\shfta,1+\shfta) {\Large $v$};
\node [color=purple] at (2.5-\shfta,1+\shfta) {\Large $p$};
\node [color=purple] at (3.5-\shfta,1+\shfta) {\Large $q$};

\end{tikzpicture}
\end{center}
In this picture, only the segment represented by the doubled lines
is guaranteed to be isometric to the corresponding segment in $\overline{U(G)}$. 

If $p$ and $q$ both lie in doubled portion (as they do in the picture
above), 
the calculation
in \Ref{ineq-proof} applies, and the desired inequality follows. In fact,  
\Ref{ineq-proof} applies unless
$p$ and $q$ both lie in the red region or both lie in the green region, since in all other cases the relevant components
of \Ref{ineq-proof} in $T$ may be computed 
using
real arithmetic in $\overline{T}$.

The remaining cases are easy to deal with, since if $p$ and $q$ both lie
in the red region or the green region, then $d(p,q)$ is bounded by
$d(x,y)$ or $d(u,v)$, as appropriate. By assumption, both 
of these distances are $\le \delta$,
and the conclusion follows.
\end{proof}

\begin{Definition}[Analog of \Ref{round-pentagon}]
Define maps 
$\lfloor \cdot \rfloor : \overline{U(G)} \to U(G)$ (rounding), and
$[\cdot]: U(G) \to G$ (residue),
as follows:
\begin{enumerate}
\item If $\bar{x} \in \overline{U(G)}$, and $\bar{x}$ lies in an edge $e$, define
$\lfloor \bar{x} \rfloor$ to be the vertex of $e$ closest to the root $\tilde{v}_0$. 
\item If $\tilde{x} \in U(G)$, define $[\tilde{x}]= u(\tilde{x})$, the projection of $\tilde{x}$
onto $G$.
\end{enumerate}
\end{Definition}
\begin{Lemma}\label{round-gen}
If $\bar{x}, \bar{y} \in \overline{U(G)}$ and $d(\bar{x},\bar{y}) \le 1$, then 
$d(\lfloor\bar{x}\rfloor, \lfloor \bar{y} \rfloor) \le 1$.  If $\tilde{x}, \tilde{y}\in U(G)$ and
$d(\tilde{x},\tilde{y}) \le 1$, then $d([\tilde{x}],[\tilde{y}) ]\le 1$.
\end{Lemma}
\begin{proof} These statements are elementary.
\end{proof}

\begin{Definition}[Analog of \Ref{sn-pentagon}]
If $\sigma: Q_d \to G$ and $\tilsig^N:  Q_d^N \to \overline{U(G)}$ is the map
defined above, define $[\tilsig^N]: Q_d^N \to G$ by
\[
[\tilsig^N] (a_1, \dots, a_d) = \big[ \lfloor \tilsig^N(a_1,\dots,a_d)\rfloor \big ].
\]
Finally, define $S^N(\sigma)$ to be the sum (in $C_d(G)$) of the small subcubes 
appearing in $[\tilsig^N]$.
\end{Definition}
The following corollary is 
easily proved using 
\Ref{ineq-gen} and \Ref{round-gen}.

\begin{Corollary}[Analog of \Ref{cor-parts12-pentagon}] 
\begin{enumerate}
\item
$[\tilsig^N]$ is a graph map from $Q_d^N$ to $G$.
\item
If $n \ge d$, then every small subcube appearing in $[\tilsig^N]$ has at most two labels. In other
words, $S^N: C_d(G) \to C_d^{(2)}(G)$.
\end{enumerate}
\end{Corollary}
To complete the proof of \Ref{thm-largegirth} we must construct maps $h_d: \CC_d(G)  \to \CC_{d+1}(G)$ and
$h_{d-1}: \CC_{d-1}(G)  \to \CC_d(G)$ and show that they define a chain homotopy.
To establish part (1) of \Ref{thm-largegirth}, the constructions of $h_d$
and $h_{d-1}$ in
Section 3 carry over with almost no change. 
The only non-trivial step is showing that 
$h_d$ and $h_{d-1}$ are defined by graph maps
on the generators of $\CC_d(G)$ and $\CC_{d-1}(G)$, but
this is a straightforward consequence
of \Ref{ineq-gen} and \Ref{round-gen}. 
Once these details are established,  the arguments proving parts (2) and (3) carry over as well, and the proof of \Ref{thm-largegirth} is complete.
\end{proof}

\section{Non-vanishing Homology in High Dimensions}
In this section we will construct an infinite sequence of graphs $\{G_d\}_{d \ge 0}$ such that 
$\Hom_d^\Cube(G_d) \ne (0)$ for all $d$. 
The construction is inspired by 
\cite[Definition 5.2]{BCW} but includes some
additional (necessary) details.

In the following definition, if $G$ and $H$ are 
graphs, $G \,\Box \, H$ denotes
the graph with vertex set $V(G)\times V(H)$ and edges
$\{(g_1,h_1), (g_2,h_2)\}$, where either
$g_1=g_2$ and $\{h_1,h_2\}\in E(H)$ or
$h_1=h_2$ and $\{g_1,g_2\}\in E(G)$.

\begin{Definition} \label{seq-def}
Let $G$ be a graph.
\begin{enumerate} 
\item If
    $N$ is a positive integer, let $G^{\times N}$ denote the
    graph $(G\, \Box \,I_{N})/\sim$ , where
    $I_{N}$ is the path with vertex set
    $\{0,\dots,N\}$ and
    $\sim$ identifies each of the 
    subgraphs $\{0\}\times G$ and $\{N\}\times G$ to 
    single points, denoted $\overline{0}$ and 
    $\overline{N}$, respectively.
\item Define the sequence
$\{G_d\}_{d\ge  1}$ by setting 
$G_1 = G$, and $G_{d+1} = G_d^{\times (d+3)}$ for
$d\ge 1$.
\end{enumerate}
\end{Definition}

\begin{Theorem} \label{thm-seq}
If $G$ is any graph, and the sequence $\{G_d\}_{d\ge  1}$ is constructed as in \Ref{seq-def} then
$ \Hom_{d+1}^\Cube(G_{d+1})= \Hom_{d}^\Cube(G_{d})$ for all $d\ge 1$.
\end{Theorem}

\begin{figure}
\begin{center}
\begin{tikzpicture}[scale =.4]
\def\shftup{4}
\coordinate (p0) at (5,0);
\coordinate (p1) at (0.1,1.5);\coordinate (p2) at (5,2);\coordinate (p3) at (10,5);\coordinate (p4) at (7,6);\coordinate (p5) at (2,5);
\coordinate (p6) at (0.1,1.5+\shftup);\coordinate (p7) at (5,2+\shftup);\coordinate (p8) at (10,5+\shftup);\coordinate (p9) at (7,6+\shftup);\coordinate (p10) at (2,5+\shftup);
\coordinate (p11) at (0.1,1.5+2*\shftup);\coordinate (p12) at (5,2+2*\shftup);\coordinate (p13) at (10,5+2*\shftup);\coordinate (p14) at (7,6+2*\shftup);\coordinate (p15) at (2,5+2*\shftup);
\coordinate (p16) at (5,16);
\draw[fill=black] (p0) circle (6pt);
\draw[fill=black] (p1) circle (6pt);\draw[fill=black] (p2) circle (6pt);\draw[fill=black] (p3) circle (6pt);\draw[fill=black] (p4) circle (6pt);\draw[fill=black] (p5) circle (6pt);
\draw[fill=black] (p6) circle (6pt);\draw[fill=black] (p7) circle (6pt);\draw[fill=black] (p8) circle (6pt);\draw[fill=black] (p9) circle (6pt);\draw[fill=black] (p10) circle (6pt);
\draw[fill=black] (p11) circle (6pt);\draw[fill=black] (p12) circle (6pt);\draw[fill=black] (p13) circle (6pt);\draw[fill=black] (p14) circle (6pt);\draw[fill=black] (p15) circle (6pt);
\draw[fill=black] (p16) circle (6pt);
\draw[thick] (p1) -- (p2) -- (p3) -- (p4) -- (p5) -- (p1) ;
\draw[thick] (p0)--(p1);\draw[thick] (p0)--(p2);\draw[thick] (p0)--(p3);\draw[thick] (p0)--(p4);\draw[thick] (p0)--(p5);
\draw[thick] (p6) -- (p7) -- (p8) -- (p9) -- (p10) -- (p6) ;
\draw[thick] (p16)--(p11);\draw[thick] (p16)--(p12);\draw[thick] (p16)--(p13);\draw[thick] (p16)--(p14);\draw[thick] (p16)--(p15);
\draw[thick] (p11) -- (p12) -- (p13) -- (p14) -- (p15) -- (p11) ;
\draw[thick] (p6)--(p1);\draw[thick] (p7)--(p2);\draw[thick] (p8)--(p3);\draw[thick] (p9)--(p4);\draw[thick] (p10)--(p5);
\draw[thick] (p6)--(p11);\draw[thick] (p7)--(p12);\draw[thick] (p8)--(p13);\draw[thick] (p9)--(p14);\draw[thick] (p10)--(p15);
\end{tikzpicture}
\end{center}
\caption{Graph $G_2 = G_1^{\times 4}$, obtained from $G_1=\ZZ_5$.}
\label{G3}
\end{figure}
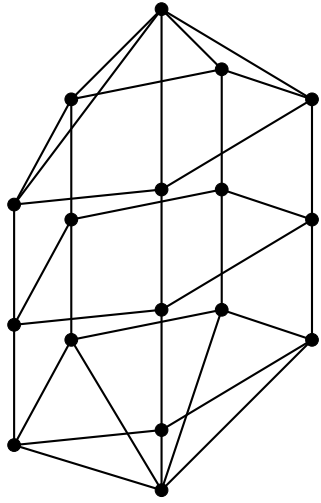

For example, if $G = G_1 = \ZZ_5$, then for the graph
$G_2$ shown in \Ref{G3}, we have $\Hom^\Cube_2(G_2) = R$.

The proof of \Ref{thm-seq} depends on the 
following lemma. 
\begin{Lemma}\label{stretch} If $G$ is any graph, and
 $\sigma: Q_d\to G^{\times N}$ is a singular
$d$-cube, then if $N\ge d+1$, the image
of $\sigma$ cannot contain both 
$\overline{0}$ and
$\overline{N}$.
\end{Lemma}
\begin{proof}
Since $\sigma$ is a graph map,  $d(\sigma(x),\sigma(y))\le 
d(x,y)$ for all $x,y\in Q_d$.  Since 
diameter($Q_d)=d$ and $d(\overline{0},\overline{N})= N$ in $G^{\times N}$, the result
follows. 
\end{proof}

\begin{proof}[Proof of \Ref{thm-seq}]
Fix $d \ge 1$, and define
$
A = G_{d+1} - \overline{0}$
and
$
B = G_{d+1} - \overline{d+3}$.
It follows from \Ref{stretch} that 
\[\CC_k(G_{d+1})=\CC_k(A) + \CC_k(B)\]
for all $k \le d+2$. This enables us
to derive a segment of the Mayer-Vietoris
sequence for $A,B$ and $G_{d+1}$. More
precisely, for $k = d, d+1, d+2$ we have short exact sequences
\begin{equation} \label{short}
\xymatrix{
0  \ar[r]  & 
\CC_k(A\cap B ) \ar[r]^>>>>>{i} & \CC_k(A)\oplus \CC_k(B)  \ar[r]^>>>>>{j} & \CC_k(G_{d+1}) \ar[r] & 0
}
\end{equation}
where $i(x) = (x,-x)$ for all $x\in A\cap B$ and $j(x,y) = x+y$ for all $x\in A,\, y \in B$. Using 
standard arguments (e.g. \cite[\S 2.1]{Hat}) one obtains the following exact sequence in homology:
\begin{align}\label{exact-hom}
\Hom^\Cube(&\CC_{d+1}(A))\oplus 
\Hom^\Cube(\CC_{d+1}(B)) 
\stackrel{j_*}{\longrightarrow}
\Hom^\Cube(\CC_{d+1}(G_{d+1})) \\ \nonumber
&
\stackrel{\Delta}{\longrightarrow}
\Hom^\Cube(\CC_{d}(A\cap B)) 
\stackrel{i_*}{\longrightarrow}
\Hom^\Cube(\CC_{d}(A)\oplus
\Hom^\Cube(\CC_{d}(B)).
\end{align}
We emphasize that \Ref{exact-hom} is not meant to be part of a full
Mayer-Vietoris sequence, since the
sequences \Ref{short} are not guaranteed
to be exact for $k > d+2$.
In \Ref{exact-hom} the map $\Delta$ may be defined
as follows: if $[x]\in
\Hom^\Cube (\CC_{d+1}(G_{d+1}))$, write 
$x = y_A + y_B$, where $y_A \in \CC_{d+1}(A)$ and
$y_B \in \CC_{d+1}(B)$. Then
$\Delta[x] = [\partial y_A] = [-\partial y_B] 
\in \CC_d(A\cap B)$.

It follows from results in 
\cite[\S 4]{BGJW} that 
$\Hom^\Cube(\CC_k(A))= \Hom^\Cube(\CC_k(B) = (0)$ for $k\ge 1$,
since $A$ and $B$ are contractible (in the
sense of \cite{BKLW}) to single-point graphs.
Furthermore,  $\Hom^\Cube(\CC_d(A\cap B)) \approx
\Hom^\Cube (\CC_d(G_{d}))$ since 
$G_d$ is a deformation retraction (in the
sense of \cite[\S 4]{BGJW}) of $A\cap B$.
Hence \Ref{exact-hom} reduces to the
exact sequence
\[
0
\stackrel{j_*}{\longrightarrow}
\Hom^\Cube(\CC_{d+1}(G_{d+1})) 
\stackrel
{\Delta}
{\longrightarrow}
\Hom^\Cube(\CC_{d}(G_d))
\stackrel
{i_*}
{\longrightarrow}
0,
\]
and the proof is complete.
\end{proof}
We do not know if our construction of $G_{d+1}$
in \Ref{seq-def} is ``tight", i.e. whether \Ref{thm-seq} would hold if we defined $G_{d+1} = G_d^{\times N}$ for some smaller
value of $N$. Our proof depends on the choice of $N = d+3$ but, for example, computation shows that if $G = G_1 = \ZZ_5$, 
then the graph $G'_2 = G_1^{\times 3}$ also has homology equal to $R$ in dimension $2$. It would be interesting to explore whether $d+3$ in \Ref{seq-def}(2) could be reduced to $d+2$ in general.

\section{Final Remarks}

The main results in this paper (and much more) could be proved more easily if we had a complete Mayer-Vietoris theory at our disposal. As in classical treatments (e.g. \cite{Hat},\cite{Mu}), the main tool in would be a 
``covering lemma" stating that if a graph $G$ can be covered (set-theoretically) by a family of subgraphs ${\mathcal K} = \{K_i\}$ satisfying appropriate ``neighborhood'' conditions, then
$\Hom^\Cube(\CC(G)) = \Hom^\Cube(\CC^{\mathcal K}(G))$, where
\[\CC_d^{\mathcal K} =
\{\sigma \in \CC_d(G) \;|\; \textrm{Im}(\sigma) \subseteq K_i \textrm{ for some $i$}\}.
\]
In fact for graphs $G$ with no $3$-cycles or
$4$-cycles, the subdivision techniques in Sections 3 and 4 of this paper provide
such a theory, but it is of little independent
value because all homology groups are trivial
in dimension $d\ge 2$.  

It is tempting to speculate that
for arbitrary graphs, the covering space arguments in Section 4 might be modified and/or extended to prove something like the following:

\begin{Conjecture}
Let $G$ be a graph, and let 
${\mathcal K} = \{K_i\}$ be a covering of $G$ by subgraphs such that every edge, $3$-cycle, and
$4$-cycle of $G$ is contained in some $K_i$. Then
$\Hom^\Cube(G) \approx \Hom^\Cube(\CC^{\mathcal K}(G))$.
\end{Conjecture}

However, this conjecture is false:  
for example, 
if $G=Q_3$, the $3$-cube, then 
$\Hom_2^\Cube(G)=(0)$ by Corollary 4.3 of 
\cite{BGJW}. On the other hand, if ${\mathcal K}$ is the covering of $G$ by its six quadrilateral faces, then direct
computation (omitted here) shows that
$\Hom_2^\Cube(\CC^{\mathcal K}(G))= R$. 

While it falls short of providing a
Mayer-Vietoris theory for 
$\Hom^\Cube(G)$, the following proposition 
does help identify what $\Hom^\Cube(\CC^\kK(G))$ is computing, and may be of independent interest.

\begin{Proposition}\label{prop-covering}
Let $G$ be a graph, and suppose that 
${\mathcal K}$ is the covering of $G$ by
its edges, $3$-cycles, and $4$-cycles.
Then 
$\Hom^\Cube(\CC^\kK(G)) \approx
\Hom^\Sing(G^*)$, where the latter denotes the 
singular homology of the cell complex $G^*$ obtained from $G$ (considered as a $1$-complex) with $3$- and $4$-cycles filled in as $2$-dimensional cells.
\end{Proposition}
\begin{proof}
  For the proof it will be convenient to consider the cellular homology $\Hom^\Cell(G^*)$ instead of $\Hom^\Sing(G^*)$. By
  \cite[Theorem 2.35]{Hat} $\Hom^\Sing(X) \approx \Hom^\Cell (X)$ for any CW-complex. In particular, it suffices to show that
  $\Hom^\Cube(\CC^\kK(G)) \approx
\Hom^\Cell(G^*)$.
  
Let $\kK^*$ be the cover of $G^*$
by its edges, triangles, and quadrangles. Note that $G^*$ has a natural CW-structure and each element of the cover is a closed subcomplex of $G^*$. First, by construction there is
a natural bijection between $\kK$ and $\kK^*$. 
For any non-empty subset $\sigma \subseteq \kK$ let $\kK_\sigma$ be the intersection 
of the elements of $\sigma$. 
Correspondingly, for any non-empty 
subset $\sigma^* \subseteq \kK^*$ let 
$\kK^*_{\sigma^*}$ be the intersection of 
the elements of $\sigma^*$.

\noindent {\sf Claim 1:} For any non-empty $\sigma \subseteq \kK$ and the 
corresponding $\sigma^* \subseteq \kK^*$ 
we have $$\Hom^\Cube_* (\kK_\sigma)
\approx \Hom^\Cell_* (\kK^*_{\sigma^*}).$$
Moreover, both are trivial in homological dimensions $\geq 1$.

\noindent {\sf Proof of Claim 1:} 
For $\sigma \subseteq \kK$ with $|\sigma| > 1$, $\kK_\sigma$ and $\kK^*_{\sigma^*}$
are both either empty, one or two points, 
or a path of length $1$ or $2$. 
On these graphs and spaces the discrete cubical and 
cellular homology theories coincide and are trivial in dimensions $\geq 1$.

For $\sigma \subseteq \kK$ with 
$|\sigma| = 1$, either both $\kK_\sigma$ and $\kK^*_{\sigma^*}$ are edges, 
$\kK_\sigma$ is a triangle graph and 
$\kK^*_{\sigma^*}$ is a solid triangle, or
$\kK_\sigma$ is a quadrangle graph and 
$\kK^*_{\sigma^*}$ is a solid quadrangle.
The cellular homology groups of a solid quadrangle and a solid triangle are trivial
in homological dimension $\geq 1$, as are the discrete cubical homology groups of
a quadrangle and a triangle graph.
This proves the claim.

\noindent {\sf Claim 2:} 
Let $\tau \subseteq \sigma \subseteq \kK$. Then the induced maps $$\Hom^\Cube_0(\kK_\sigma) 
\rightarrow \Hom^\Cube_0(\kK_\tau) \text{~and~} 
\Hom^\Cell_0(\kK^*_{\sigma^*}) 
\rightarrow \Hom^\Cell_0(\kK^*_{\tau^*})$$ are either
both isomorphisms, both $0$ maps, or both are projections 
of a rank $2$ onto a rank $1$ homology group.

\noindent {\sf Proof of Claim 2:} It follows
from an analysis of the proof of the previous claim that all homology groups in dimension $0$
are of rank $1$ unless the intersection is empty (in which case it is $0$) or the intersection is a two point set (in which case
it is of rank $2$). Now the assertion follows
by inspecting the cases.

Let $N(\kK)$ be the nerve of the covering
$\kK$, that is, the simplicial complex whose $p$-simplices are the subcollections of $\kK$ of size $p+1$ with non-empty intersection.
For $p\ge 0$, let $N^{(p)}(\kK)$ denote the set of faces of dimension $p$
in $N(\kK)$. Note that by the above arguments the
nerve is the same for $\kK$ and $\kK^*$. From the given data we build, in the usual way (see, e.g., \cite[Chap. 7]{KBrown}), two double complexes 

\[
C_{p,q}^\Cube(G) = \bigoplus_{\sigma \in N^{(p)}(\kK)} C^\Cube_q(\kK_\sigma)
\]

and
\[
C_{p,q}^\Cell(G) = \bigoplus_{\sigma \in N^{(p)}(\kK)} C^\Cell_q(\kK^*_{\sigma^*})
\]
where for $\bullet \in \{\Cube,\Cell\}$ the 
vertical differential $C_{p,q}^\bullet \longrightarrow C_{p,q-1}^\bullet$ is the differential in the cubical or cellular homology theory, and
the horizontal differential $C_{p,q}^\bullet \longrightarrow C_{p-1,q}^\bullet$ is induced by the inclusion maps combined with the
sign of the differential of the nerve,
seen as a simplicial complex. 
More precisely, let $\sigma\in N^{(p)}(\kK)$, $\tau \subseteq \sigma$
and  $|\tau| \in N^{(p-1)}(\kK)$. 
It follows that $\kK_\sigma \subseteq \kK_\tau$ and $\tau = \sigma \setminus \{j\}$ for some $j \in \sigma$.
Thus $\tau$ appears in the simplicial
boundary of $\sigma$ with some sign $\epsilon_{\sigma,\tau}$. 
Now the differential $C_{p,q}^\bullet \longrightarrow C_{p-1,q}^\bullet$ is the sum of the maps $C^\bullet_q(\kK_\sigma) \rightarrow C^\bullet_q(\kK_\tau)$, for all $\tau \subseteq \sigma$, $|\tau| \in N^{(p-1)}(\kK)$, induced
by inclusion and multiplied by $\epsilon_{\sigma,\tau}$. 

The two double complexes yield with respect to the vertical differential a spectral sequence (see, e.g.,  \cite[Chapter 7]{KBrown}) with $E^1$ page given by 
$E_{p,q}^1 (\Cube)=  \bigoplus_{\sigma\in N^{(p)}(\kK)} \Hom^\Cube_q(\kK_\sigma)$
and
$E_{p,q}^1 (\Cell) =  \bigoplus_{\sigma\in N^{(p)}(\kK)} \Hom^\Cell_q(\kK^*_{\sigma^*})$. 
From the first claim we already know that
$E_{p,q}^1(\Cube) \approx E_{p,q}^1 (\Cell)$ 
and $E_{p,q}^1(\Cube) =
E_{p,q}^1(\Cell) = 0$ for $q \geq 1$.  

Next we consider the differentials of 
the $E^1$ pages in both cases.
Note that for $\bullet \in \{\Cube,\Cell\}$
the differential of 
$E_{p,q}^1(\bullet) \rightarrow 
E_{p-1,q}^1(\bullet)$ is
induced as follows. Consider
$\sigma \in N^{(p)}$  and let $\tau_0,\ldots,\tau_{p}$ be its
maximal faces.
Then we have a map in homology 
$$\Hom_{q}^\bullet(\kK_\sigma) \rightarrow
\bigoplus_{i=0}^{p-1} \Hom_q^\bullet(
\kK_{\tau_i})$$ induced by the sum of the inclusion maps weighted with the sign of
the simplicial differential. This map in turn induces the differential on $E_{p,q}^1(\bullet)$. By the second claim the maps in homology (if they are not trivial) coincide for both
theories. As a consequence the differentials
on the $E^1$-pages are the same and so 
$E_{p,q}^2(\Cube) \approx E_{p,q}^2(\Cell)$.

Since $E_{p,q}^1(\Cube) = E_{p,q}^1(\Cell) \neq 0$ implies $q=0$
the same holds true for $E_{p,q}^s(\bullet)$ and 
$s \geq 2$. But the differential on the 
$s$\textsuperscript{th} page maps  
$E_{p,q}^s(\bullet) \rightarrow E_{p-s,q+1-s}^s(\bullet)$.
This shows that the differential of the
$s$\textsuperscript{th} page is trivial for all $s \geq 2$.  Thus both sequences
have the same limit and the limit is obtained
on the $E^2$-page, with non-zero entries concentrated on the bottom row.

For any double complex $\{C_{p,q}\}$, one can construct not only a spectral sequence $\{E^r\}$, but also a second (transposed) spectral sequence 
$\{\tilde{E}^r\}$ using horizontal differentials (instead of vertical)
to define $\tilde{E}^1$, and then constructing subsequent pages by the
usual rules. In this case, it is well known (e.g. \cite[Chap. 5]{Weibel}) that both $\{E^r\}$ and $\{\tilde{E}^r\}$ compute the same homology, namely, the homology of the total complex $\{TC_d\}_{d\ge 0}$,
where $TC_d = \bigoplus_{p+q=d} C_{p,q}$.

In our situation, it can be shown
(see \cite[Chap. 7]{KBrown} for an example of the argument) that 
$\{\tilde{E}^r(\Cube)\}$ and $\{\tilde{E}^r(\Cell)\}$ both converge on
the second page $\tilde{E}^2$, with non-zero entries concentrated in the first column. In that column we have
\[
\tilde{E}^2_{0,d}(\Cube) \approx \Hom^\Cube_d(C^\kK(G))
\textrm{  and  }
\tilde{E}^2_{0,d}(\Cell) \approx \Hom^\Cell_d(G^*)
\]
for all $d\ge 0$.
Combining all of the above information, we obtain,  for all $d\ge 0$,
\begin{align*}
\Hom_d^\Cube(C^{\kK}(G)) \approx
\tilde{E}^2_{0,d}(\Cube) \approx
E^2_{d,0}(\Cube) &\\
\approx  E^2_{d,0}(\Cell)& \approx
\tilde{E}^2_{0,d}(\Cell) \approx
\Hom_d^\Cell(G^*)  
\end{align*}
as claimed, and the proof is complete.
\end{proof}

As a somewhat non-trivial illustration of \Ref{prop-covering}, let $G$ be the $1$-skeleton of an octahedron, as illustrated in 
\Ref{fig-oct}. Here $\kK$ is the covering of $G$ consisting
of 12 edges, 8 triangles, and 3 quadrangles. Computation (omitted) shows that $\Hom^\Cube_2(C^\kK(G)) = \Hom^\Sing_2(G^*) = R^4$.
On the other hand, it follows from Corollary 4.6 in \cite{BGJW} (or can be shown directly) that
$\Hom_2^\Cube(\CC(G)) = (0)$.

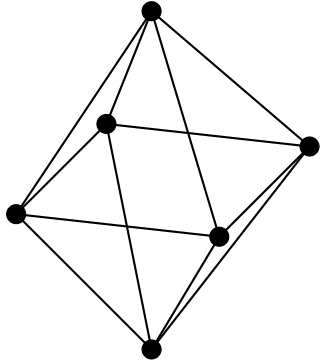
\begin{figure}
\begin{center}
\begin{tikzpicture}[scale =.6]
\def\shftup{4}

\coordinate (p0) at (3.5,-.5);
\coordinate (p1) at (.5,2.5);\coordinate (p2) at (5,2);\coordinate (p4) at (2.5,4.5);\coordinate (p3) at (7,4);\coordinate (p5) at (3.5,7);
\draw[fill=black] (p0) circle (6pt);
\draw[fill=black] (p1) circle (6pt);\draw[fill=black] (p2) circle (6pt);\draw[fill=black] (p3) circle (6pt);\draw[fill=black] (p4) circle (6pt);\draw[fill=black] (p5) circle (6pt);
\draw[thick] (p1) -- (p2) -- (p3) -- (p4) -- (p1) ;
\draw[thick] (p0)--(p1);\draw[thick] (p0)--(p2);\draw[thick] (p0)--(p3);\draw[thick] (p0)--(p4);
\draw[thick] (p5)--(p1);\draw[thick] (p5)--(p2);\draw[thick] (p5)--(p3);\draw[thick] (p5)--(p4);
\end{tikzpicture}
\end{center}
\caption{Graph $G$ = $1$-skeleton of octahedron.}
\label{fig-oct}
\end{figure}

Our results and observations highlight the 
special role played by graphs without 
$3$- and $4$-cycles, and indeed, for these graphs, the discrete and ordinary singular cubical homology theories agree completely. It would be interesting to establish the precise connections between the two theories for more general
graphs.
While results like \Ref{prop-covering} might provide information about 
$\Hom^\Cube(C^\kK(G))$ for a covering $\kK$, they do not give us 
information about $\Hom^\Cube(G)$ directly, and a richer theory is needed to help bridge this gap.

\section{Acknowledgements}
We are grateful to Eric Babson for several comments that helped clarify our proofs, especially of
\Ref{corr-pentagon} and \Ref{lift-pentagon}.

\bibliographystyle{amsplain}
\bibliography{DiscHomo}
\end{document}